\documentclass[11pt,reqno]{amsart}
\usepackage{amssymb}
\usepackage{amsmath}
\usepackage{amsthm}
\usepackage{mathrsfs}
\usepackage{url}
\usepackage[all]{xy}
\usepackage{ifthen}
\usepackage{pstricks}
\usepackage{pst-plot}
\usepackage{rotating}
\usepackage{graphicx}

\newcommand{\A}{\ensuremath{\mathbb{A}}}
\newcommand{\cA}{\ensuremath{\mathcal{A}}}

\newcommand{\C}{\ensuremath{\mathbb{C}}}

\newcommand{\sD}{\ensuremath{\mathscr{D}}}

\newcommand{\cF}{\ensuremath{\mathcal{F}}}
\newcommand{\G}{\ensuremath{\mathbb{G}}}

\newcommand{\cH}{\ensuremath{\mathcal{H}}}
\newcommand{\rH}{\ensuremath{\mathrm{H}}}
\newcommand{\bL}{\ensuremath{\mathbb{L}}}

\newcommand{\sM}{\ensuremath{\mathscr{M}}}

\newcommand{\cO}{\ensuremath{\mathcal{O}}}

\newcommand{\bP}{\ensuremath{\mathbb{P}}}
\newcommand{\cP}{\ensuremath{\mathcal{P}}}

\newcommand{\Q}{\ensuremath{\mathbb{Q}}}

\newcommand{\cT}{\ensuremath{\mathcal{T}}}

\newcommand{\cU}{\ensuremath{\mathcal{U}}}

\newcommand{\Z}{\ensuremath{\mathbb{Z}}}
\newcommand{\cZ}{\ensuremath{\mathcal{Z}}}

\newcommand{\iso}{\cong}

\newcommand{\into}{\hookrightarrow}

\newcommand{\CH}{\ensuremath{\mathrm{CH}}}

\newcommand{\Smk}{\mathrm{Sm_k}}

\newcommand{\alt}{\mathrm{alt}}

\DeclareMathOperator{\Cell}{Cell}

\DeclareMathOperator*{\colim}{colim}
\DeclareMathOperator{\Comod}{Comod}

\DeclareMathOperator{\Ext}{Ext}
\DeclareMathOperator{\gr}{gr}
\DeclareMathOperator{\Gr}{Gr}
\DeclareMathOperator{\Hom}{Hom}

\DeclareMathOperator{\Mod}{Mod}
\DeclareMathOperator{\Pic}{Pic}

\DeclareMathOperator{\Rep}{Rep}
\DeclareMathOperator{\sgn}{sgn}
\DeclareMathOperator{\Spec}{Spec}

\DeclareMathOperator{\Vect}{Vect}

\numberwithin{equation}{subsection}
\numberwithin{figure}{subsection}

\newtheorem{thm}{Theorem}[subsection]
\newtheorem{prop}[thm]{Proposition}
\newtheorem{lemma}[thm]{Lemma}
\newtheorem{cor}[thm]{Corollary}

\newtheorem{conj}[thm]{Conjecture}

\theoremstyle{definition}
\newtheorem{defn}[thm]{Definition}
\newtheorem{rmk}[thm]{Remark}
\newtheorem{eg}[thm]{Example}
\newtheorem{qst}[thm]{Question}
\newtheorem{const}[thm]{Construction}

\newcommand{\NonDraftVersion}{10}
\newcommand{\PaperVersion}{\NonDraftVersion}  

\begin{document}

\title[The motivic fundamental group of $\bP^1-\{0,1,\infty\}$]{The motivic fundamental group of the punctured projective line}
\author{Bertrand J. Guillou}
\address{Department of Mathematics, University of Illinois at Urbana-Champaign, 1409 W. Green Street, Urbana, IL 61801}
\email{bertg@illinois.edu}

\begin{abstract} We describe a construction of an object associated to the fundamental group of $\bP^1-\{0,1,\infty\}$ in the Bloch-Kriz category of mixed Tate motives.
This description involves Massey products of Steinberg symbols in the motivic cohomology of the ground field. 
\end{abstract}
\date{\today}
\maketitle



\tableofcontents










\section{Introduction}



The importance of the algebraic fundamental group of $\bP^1-\{0,1,\infty\}$ has been known for some time:
\begin{quote}
``[A. Grothendieck] m'a aussi dit, avec force, que le compl\'{e}t\'{e} profini $\hat\pi_1$ du groupe fondamental de $X:=\bP^1(\C)-\{0,1,\infty\}$, avec son action de $\mathrm{Gal}(\overline{\Q}/\Q)$ est un objet remarquable, et qu'il faudrait l'\'{e}tudier.'' \qquad\qquad -\cite{DelP1}
\end{quote}
Indeed, Belyi's theorem implies that the canonical action of $\mathrm{Gal}(\overline{\Q}/\Q)$ is faithful. P. Deligne initiated the study of this fundamental group from the motivic viewpoint in \cite{DelP1}. The theory of motives is a theory of cohomology, and one knows from rational homotopy theory that 
one does not expect to recover arbitrary fundamental groups from cohomological information. Rather, one expects only to recover their nilpotent completions. Thus we seek to describe the quotients $\Q[\pi_1(X)]/I^n$, where $I$ denotes the augmentation ideal. Deligne and Goncharov have recently (\cite{DG}) used their description of this motivic fundamental group to deduce bounds on the dimensions of $\Q$-vector spaces generated by multiple zeta values. We give a description of the motivic torsor of paths in the Bloch-Kriz-May setup of mixed Tate motives. 

\vspace{4mm}
Let $X$ be a topological space and fix $a,b\in X$. The diagonal map on $X$ makes it into a (cocommutative) comonoid, and $X$ coacts on the two chosen basepoints, meaning that we have maps $a:*\to *\times X$ and $b:*\to X\times *$. We can thus form the two-sided cobar construction $coB^\bullet(a,X,b)$.  Recall that this is a cosimplicial space given in degree $n$ by 
\[coB^n(a,X,b)=\,\!_a X^n_b:=\{a\}\times X^n\times\{b\}\iso X^n.\]
The first and last coface maps $\,\!_aX^n_b\to\,\!_a X^{n+1}_b$ use the coaction of $X$ on $a$ and $b$, respectively, while the remaining coface maps are defined using the diagonal on $X$. The codegeneracies are given by projections.

The totalization of this cosimplicial space gives a model for the space $\,_b\cP_a$ of paths in $X$ from $a$ to $b$. 
Let $\,_bP_a=\pi_0(\,_b\cP_a)$ be the set of homotopy classes of paths, and note that it is a $\pi_1(X,b)-\pi_1(X,a)$-bimodule. Let $I$ denote the augmentation ideal in, say, $\Q[\pi_1(X,b)]$, and suppose $a\neq b$. Following Beilinson, Deligne and Goncharov (loc. cit.,  Proposition~3.4) show, under some finiteness hypotheses on $X$, that (the dual of) $\Q[\,_bP_a]/I^{n+1}$ is given by the relative cohomology group $\rH^n(X^n,\text{cofaces})$. The complex
\begin{equation}
C^*(X^n)\to C^*(a\times X^{n-1})\oplus 
\dots\oplus C^*(X^{n-1}\times b)\to\dots\to C^*(pt)^{\oplus^{n+1}}
\label{theComplex}\end{equation}
calculates this relative cohomology group.

The point is that this description can quite readily be mimicked in the motivic context. The definition of the cosimplicial scheme goes through without any changes, and one wants to study the above complex, replacing the singular cochains $C^*$ by a complex computing motivic cohomology. Since the coface maps in the cosimplicial scheme are given by closed inclusions (in particular, they are not flat maps), Bloch's cycle complex will not do. Rather, we must use a functorial model for motivic cochains. The above complex then defines an object of the derived category $\sD(\cA)$, where $\cA$ is the dga of motivic cohains on $\Spec k$. Using the notation of \S\ref{KMSect}, Deligne and Goncharov show 
that it lies in $\sD(\cA)^{\geq 0}$. Restricting now to $X=\bP^1-\{0,1,\infty\}$, we will see that it also lies in $\sD(\cA)^{\leq 0}$, so that it lies in the intersection $\cH$ and thus defines a mixed Tate motive in the sense of Kriz and May.
Concretely, our main task will be to describe a cellular approximation of 
this complex in the case $X=\bP^1-\{0,1,\infty\}$.

The paper is organized as follows. The sections  \ref{MotCohomSec} and \ref{MTMSec} serve as background and fix some notations. More specifically, \S\ref{MotCohomSec} reviews relevant properties of motivic cohomology and specifies a model for motivic cochains, while \S\ref{MTMSec} reviews the theory of mixed Tate motives, discussing in particular the approaches of Bloch-Kriz (\cite{BK}) and Kriz-May (\cite{KM}). Section~\ref{CellSec} contains the body of the paper. In \S\ref{n2case}, we begin to consider the problem in the case $n=2$. Here we see that the Steinberg relation in motivic cohomology (see \S\ref{ConjPicSec}) allows one to build the cell module. More precisely, the data needed to construct the cell module is a cochain bounding the Steinberg symbols $[a]\cdot[1-a]$. This study is continued in \S\ref{n3case}, where we consider the case $n=3$. Already here we see that a higher Steinberg relation is required to hold--namely, certain Massey products must contain zero. For this reason, we briefly review Massey products in \S\ref{Masseysec}. Finally, we state and prove the general answer in \S\ref{gensec}; the main result is that the extra data required to build the cell module are cochains bounding certain $n$-fold Massey products in the motivic cohomology of $\bP^1-\{0,1,\infty\}$. Strictly speaking, sections \ref{n2case} and \ref{n3case} are not necessary in order to understand the results in \S\ref{gensec}, though the reader may find it helpful to have these lower-dimesional examples worked out explicitly. 
In \S\ref{TrimFatSect}, we find minimal models for these cell modules; this yields objects of the Bloch-Kriz category of mixted Tate motives.
Finally, we close in \S\ref{hightotsec} and \S\ref{34MasseyProds} with some discussion of bounding cochains for these Massey products in certain special cases.



I am very grateful to Spencer Bloch for bringing this material to my attention and for many helpful discussions. This work owes much to the letter \cite{B}; indeed, much of section~\ref{n2case} comes directly from that source. I would also like to thank Dan Isaksen for his encouragement and advice and Igor Kriz and Marc Levine for useful conversations. The guidance and support of my advisor, Peter May, was a very strong motivating force, and I wish to give him my heartfelt thanks. This work was completed as part of my 2008 PhD thesis.




\newpage


\section{Motivic Cohomology}\label{MotCohomSec}

\begin{quote}``Imagine a world in which the $K$-theory $K^*(X)$ of a topological space $X$ had been defined, but the ordinary cohomology groups $\rH^*(X)$ had not yet been discovered \dots  Present day algebraic geometry is such a world.'' 
\begin{flushright} - \cite{BMS} \end{flushright}
\end{quote}

The above quote is from the mid 1980's, and in fact present day algebraic geometry is no longer ``such a world.'' There are now a number of constructions of a so-called ``motivic cohomology'' theory. These are known to agree when working over a perfect ground field. In this section we discuss some properties and constructions of motivic cohomology.

\subsection{The conjectural picture}\label{ConjPicSec}


We will give some constructions of motivic cohomology below, but first we describe some of the structure (both known and desired). We will work always in the category $\Smk$ of smooth varieties over a field $k$. For any abelian group $A$, motivic cohomology gives a contravariant functor 
\[\rH^{*,*}(-;A):\Smk^{op}\to \Gr\Gr Ab\]
to bigraded abelian groups. When $A=R$ is moreover a commutative ring, then $\rH^{*,*}(X;R)$ is a graded-commutative ring (the first grading contributes a sign to commutativity, but the second grading does not). For any $X$, the groups $\rH^{p,q}(X;\Z)$ vanish for $q<0$ or $p>2q$.
See Figure~\ref{CohomPtFig} for the conjectural picture of the motivic cohomology of a point.



\begin{figure}
\begin{center}
\ifthenelse{\PaperVersion=\NonDraftVersion}{
\psset{xunit=8ex}
\begin{pspicture}(-3,-3)(5,5)
\pspolygon[fillstyle=solid,fillcolor=lightgray](-3,-3)(-3,0)(5,0)(5,-3)
\pspolygon[fillstyle=solid,fillcolor=lightgray](1,0)(1,1)(2,1)(2,2)(3,2)(3,3)(4,3)(4,4)(5,4)(5,0)
\pspolygon[fillstyle=solid,fillcolor=lightgray](-3,0)(-3,2)(1,2)(1,1)(0,1)(0,0)
\pspolygon[fillstyle=vlines](-3,2)(-3,5)(1,5)(1,2)
\psgrid[subgriddiv=0,griddots=8,gridlabels=7pt](0,0)(-3,-3)(5,5)
\psaxes[linewidth=1pt,%
  	ticks=none,%
  	labels=none]{->}(0,0)(-3,-3)(5,5)
\uput{.5}[45](0.2,0){\large$\Z$}
\uput{.5}[45](1.1,.9){\large$k^\times$}
\uput{.1}[45](2,2.1){$K_2^M(k)$}
\uput{.1}[45](3,3.1){$K_3^M(k)$}
\uput{.1}[45](4,4.1){$K_4^M(k)$}
\put(4.1,-1.2){\huge$0$}
\put(-3.2,-1.2){\huge$0$}
\put(-2.5,2.7){\huge$0$?}

\end{pspicture} }{} 
\caption{$\rH^{*,*}(\Spec k)$}\label{CohomPtFig}
\end{center}
\end{figure}

Part of this conjectural picture is the

\begin{conj}[Beilinson-Soul\'{e} Vanishing Conjecture]\label{BSConj}
\[\rH^{p,q}(\Spec{k},\Z)=0 \quad \text{ if $p<0$ or if $p=0$ and $q>0$.}\]
\end{conj}

The conjectured vanishing is represented in Figure~\ref{CohomPtFig} by diagonal lines. The gray, shaded region is known to vanish. 

An important property of motivic cohomology is homotopy invariance: for \mbox{$X\in\Smk$,} if $p:X\times\A^1\to X$ denotes the projection, then the induced map
\[\rH^{*,*}(X;\Z)\xrightarrow{p^*}\rH^{*,*}(X\times\A^1;\Z)\]
is an isomorphism.

Another important property of motivic cohomology is its relationship to algebraic $K$-theory. There is an isomorphism
\[K_n(X)\otimes \Q \iso\bigoplus_q\rH^{2q-n,q}(X;\Q).\]
The space $\rH^{p,q}(X;\Q)$ can be recoved from $K_n(X)\otimes \Q$ as an eigenspace with respect to Adams operations. Moreover, one has integrally an Atiyah-Hirzebruch spectral sequence (\cite{FS}); when tensored with $\Q$, it degenerates at $E_2$ to yield the above isomorphism.

Finally, the groups $\rH^{p,q}(X;\Z)$ are completely understood for $q=0,1$. When $q=0$, the only nonzero group is $\rH^{0,0}(X;\Z)$, which is the group of locally constant $\Z$-valued functions on $X$. In particular, for $X$ connected one has $\rH^{0,0}(X;\Z)\iso\Z$. When $q=1$, the only nonzero groups are
\[\rH^{1,1}(X;\Z)\iso\cO^*(X),\] 
the group of invertible functions on $X$, and 
\[\rH^{2,1}(X;\Z)\iso\Pic(X),\]
the group of lines bundles on $X$. Of prime importance in \S\ref{CellSec} will be the Steinberg relation: given $f\in\cO^*(X)$ such that $1-f$ is also invertible on $X$, one has
\addtocounter{thm}{1}
\begin{equation} [f]\cdot[1-f]=0  \tag{\thethm}  \label{SteinbergRltn}\end{equation}
in $\rH^{2,2}(X;\Z)$. 

\begin{eg}\label{CohomNumFldEg} One of the instances in which one has a complete picture is in the case of number fields. Let $[k:\Q]<\infty$. Borel (\cite{Bor}) computed the rationalized $K$-groups $K_i(k)_\Q:=K_i(k)\otimes\Q$. They are given by
\[K_i(k)_\Q\iso\left\{\begin{array}{ll} \Q & i=0, \\ k^\times\otimes\Q & i=1, \\ \Q^{r_1+r_2} & i\equiv1\pmod4, \quad i>1, \\ \Q^{r_2} & i\equiv3\pmod4,\\ 0 & i\equiv0,2\pmod4, \quad i>1,\end{array}
\right.\]
where $r_1$ ($r_2$) denotes the number of real (resp. complex) embeddings of $k$. Moreover, it is only the summand $\rH^{1,n}(k;\Q)$ that contributes towards $K_{2n-1}(k)_\Q$. Note that this implies that the Beilinson-Soul\'{e} conjecture (\ref{BSConj}) holds for number fields.
\end{eg}

\begin{eg}\label{CohomQEg}
The above result says that, in particular, the nonzero rational $K$-groups of $\Q$ are $K_0(\Q)_\Q=\Q$, $K_1(\Q)=\Q^\times\otimes\Q$, and $K_i(\Q)=\Q$ for $i=5,9,13,\dots$. These correspond to the picture for $\rH^{p,q}(\Q;\Q)$ given by Figure~\ref{CohomQFig}. The answer for $\rH^{p,q}(\Z;\Q)$ is the same, except that $\rH^{1,1}(\Z;\Q)=0$.
\end{eg}

\begin{figure}
\begin{center}
\ifthenelse{\PaperVersion=\NonDraftVersion}{
\psset{xunit=6ex}
\begin{pspicture*}(-2,-1)(5,6)
\psgrid[subgriddiv=0,griddots=8,gridlabels=12pt](0,0)(-2,-1)(5,6)
\psaxes[linewidth=1pt,%
  	ticks=none,%
  	labels=none]{->}(0,0)(-2,-1)(5,6)
\uput{.5}[45](0.1,0){\large$\Q$}
\uput{.5}[45](1,.8){\large$\Q^\times_\Q$}
\uput{.1}[45](1.2,3.1){\large$\Q$}
\uput{.1}[45](1.2,5.1){\large$\Q$}

\end{pspicture*} }{} 
\caption{$\rH^{*,*}(\Spec \Q;\Q)$}\label{CohomQFig}
\end{center}
\end{figure}

\subsection{Bloch's cycle complex}\label{BlchCyclSect}

One of the first constructions deserving to be called motivic cohomology is due to S. Bloch.
Recall the algebraic standard cosimplicial scheme $\Delta^\bullet$: we define $\Delta^n\subseteq\A^{n+1}$ to be the closed subvariety defined by
\[\Delta^n=\{(x_0,\dots,x_n)\in\A^{n+1}\mid \sum_i x_i=1\}.\]
The usual formulae for the face maps and degeneracy maps from topology make $\Delta^\bullet$ into a cosimplicial scheme. Also, note that $\Delta^n\iso\A^n$.

\begin{defn} For $Y$ a quasi-projective variety over a field $k$, \textbf{Bloch's cycle complex} is defined as follows. Let 
$z^q(Y,n)$ be the free abelian group on the codimension $q$ subvarieties of $Y\times\Delta^n$ which meet all faces 
\[Y\times\Delta^m\subseteq Y\times\Delta^n\]
 properly, meaning that the intersection is either empty or of codimension $q$. We will call such cycles {\bf admissible}. For fixed $q$, the groups $z^q(Y,\bullet)$ form a simplicial abelian group and thus a chain complex in the usual way. Writing 
\[N^p(Y)(q)=z^q(Y,2q-p),\]
this makes $N^*(Y)(q)$ into a cochain complex.
The cohomology groups of this cochain complex, denoted $\mathrm{CH}^q(Y,*)$, are called the \textbf{higher Chow groups} of $Y$ of codimension $q$. More precisely, 
\[\mathrm{CH}^q(Y,p)=\rH^{2q-p}(N^*(Y)(q)).\]
\end{defn}

\subsubsection{Cubical variant}

For the purpose of defining products, it is more convenient to work with a cubical variant. Let $\square^n=(\A^1)^n\iso\A^n$. For $1\leq j\leq n$ and $\epsilon\in\{0,1\}$, there is an inclusion $i_{j,\epsilon}:\square^{n-1}\into\square^n$ given by inserting $\epsilon$ in the $j$th coordinate. In addition, for each $1\leq j\leq n-1$, there is a projection $\pi_j:\square^n\to\square^{n-1}$ which omits the $j$th coordinate. 

\begin{defn} Let $\tilde{N}^p(Y)^c(q)$ be the free abelian group on the codimension $q$ subvarieties of $Y\times\square^{2q-p}$ which meet all faces 
\[Y\times\square^m\subseteq Y\times\square^{2q-p}\]
 properly. Let 
\[D^p(Y)(q)\subseteq\tilde{N}^p(Y)^c(q)\]
 be the subgroup of ``degenerate cycles'', i.e., the sum of the images of the projections $\pi_j$. Then we define the \textbf{cubical Bloch complex} to be 
\[N^p(Y)^c(q)=\tilde{N}^p(Y)^c(q)/D^p(Y)(q).\]
\end{defn}

The following proposition says that the cubical cycle complexes also compute the higher Chow groups.

\begin{prop}[\cite{L1}, Thm.~4.7; \cite{BK}, Prop.~5.1]There is a canonical quasi-isomorphism 
\[N^p(Y)^c(q)\xrightarrow{\sim}N^p(Y)(q).\]
\end{prop}

\subsubsection{Products}

If $W\subseteq X\times\square^m$ and $Z\subseteq Y\times\square^n$ meet all faces properly, then the product 
\[W\times Z\subseteq X\times\square^m\times Y\times\square^n\iso X\times Y\times\square^{m+n}\]
also meets all faces properly. Thus we get a product
\[N^{2p-m}(X)^c(p)\otimes N^{2q-n}(Y)^c(q)\to N^{2(p+q)-(m+n)}(X\times Y)^c(p+q)\]
inducing a product
\[\CH^p(X,m)\otimes\CH^q(Y,n)\to\CH^{p+q}(X\times Y,m+n).\]
Although the cycle complexes only have contravariant functoriality with respect to {\em flat} maps, the higher Chow groups nevertheless have contravariant functoriality with respect to {\em all} morphisms when the codomain is smooth (\cite{B1}, Thm.~4.1). Thus, if $X$ is smooth, one can pull back along the diagonal $X\into X\times X$ to make $\CH^*(X,*)$ into a bigraded ring. Moreover, this product is graded commutative (\cite{B1}, Cor.~5.7) with respect to the homological, or cubical, grading (as opposed to the grading by codimension).

\subsubsection{Alternating cycle complex}\label{altcyclsec}
When $Y=\Spec{k}$, the cubical complex $N^p(Y)^c(q)$ yields a homotopy-
commutative dga. After tensoring with $\Q$, one can produce an honest cdga by restricting to {\em alternating cycles}, as we explain below.

Note that the symmetric group $\Sigma_n$ acts on $\square^n$ and therefore on $N^{2q-n}(Y)^c(q)$. One can then consider the subgroup of cycles which are alternating with respect to this action, that is, the cycles $Z$ such that 
\[\sigma(Z)=\sgn(\sigma)\cdot Z\] 
for all $\sigma\in\Sigma_n$. In fact, one can show (\cite{B2}, Lemma~1.1) that the alternating cycles are closed under the differential, so that the alternating cycles form a subcomplex. However, the product of alternating cycles need no longer be alternating. One can rectify this, by projecting onto the subspace of alternating cycles, but it is necessary to pass to $\Q$ coefficients. Let $\mathrm{alt}_n\in\Q[\Sigma_n]$ be the element
\[\mathrm{alt}_n=\frac1{n!}\sum_{\sigma\in\Sigma_n}\sgn{(\sigma)}\cdot\sigma.\]
Then, letting
\[A^p(Y)(q)\subseteq N^p(Y)^c(q)\otimes\Q\]
be the vector subspace of alternating cycles, mutliplication with $\mathrm{alt}_{2q-p}$ gives a projection
\[\mathrm{alt}_{2q-p}:N^p(Y)^c(q)\otimes\Q\to A^p(Y)(q).\]
We now define a product by the composition
\[A^m(X)(p)\otimes A^n(Y)(q)\to N^{m+n}(X\times Y)(p+q)\xrightarrow{\mathrm{alt}}A^{m+n}(X\times Y)(p+q).\]
When $X=Y=\Spec k$, this makes $A=A^*(\Spec k)(*)$ into a commutative dga by construction.

\begin{defn} We will define (rational) motivic cohomology to be 
\[\rH^{p,q}(X;\Q):=\rH^p(A^*(X)(q)).\]
\end{defn}

These motivic cohomology groups can be shown to agree with the higher Chow groups, using homotopy invariance of the latter.

\begin{prop}[\cite{BK}, Proposition~5.1] 
\[\rH^{p,q}(X;\Q)\iso \rH^{p}(N^*(X)^c(q)\otimes\Q)\iso \mathrm{CH}^q(X,2q-p)\otimes\Q.\]
\end{prop}

\subsection{Friedlander-Suslin-Voevodsky variant}\label{FSVSect}

As we have already mentioned, the higher Chow groups have good functoriality for maps between smooth schemes, but the cycle complexes we have introduced do not. In this section, we discuss cycle complexes of Friedlander-Suslin-Voevodsky with good functoriality. A similar discussion appears in \S4.2 of \cite{EL}. Assuming resolution of singularities, as do we, they show (\cite{EL}, proof of Lemma~4.2.1) that their presheaf of cdga's agrees with ours up to quasi-isomorphism.

\begin{defn}Given smooth schemes $X$ and $Y$, let $z_{equi}(Y,r)(X)$ denote the free abelian group on the closed and irreducible subvarieties $Z\subseteq X\times Y$ which are dominant and equidimensional of relative dimension $r$ over $X$.
\end{defn}

As is discussed in (\cite{MVW}, 16.1), for fixed $Y$ and $r$ this defines a presheaf in $X$.

\begin{defn}We will write 
\[\cZ^p(Y)(q)=C_{2q-p}(z_{equi}(\A^q,0))(Y).\]
Again, this is the
free abelian group on the (codimension $q$) closed, irreducible subvarieties of $\Delta^{2q-p}\times Y\times\A^q$ which are quasifinite and dominant over $\Delta^{2q-p}\times Y$. For fixed $q$, this defines a cochain complex.
\end{defn}

Since $C_{2q-p}(z_{equi}(\A^q,0))$ determines a presheaf on the category of smooth $k$-schemes $\Smk$ (and in fact an \'{e}tale sheaf), we get functoriality of these complexes with respect to all maps between smooth schemes.

Note that any codimension $q$ cycle $Z\subseteq \Delta^{2q-p}\times Y\times\A^q$ which is quasifinite over $\Delta^{2q-p}\times Y$ must meet the faces $\Delta^m\times Y\times\A^q$ properly, so we have an inclusion
\[\cZ^p(Y)(q)\into N^p(Y\times\A^q)(q).\]
Moreover, homotopy invariance of the higher Chow groups (\cite{B1}, Thm.~2.1) says that the pullback along the projection $Y\times\A^q\to Y$ induces a quasi-isomorphism
\[N^p(Y)(q)\xrightarrow{\sim}N^p(Y\times\A^q)(q).\]
Thus, in order to show that our new complexes compute the higher Chow groups, it remains to show that the inclusion $\cZ(Y)\into N(Y\times\A^q)$ is a quasi-isomorphism. For $Y=\Spec{k}$, this is given by the following theorem of Suslin.

\begin{thm}[\cite{S}, Thm.~2.1]The inclusion 
\[\cZ^*(\Spec{k})(q)\into N^*(\A^q)(q)\]
is a quasi-isomorphism.
\end{thm}

Note that, generalizing what we discussed above, there is an inclusion of groups
\[C_{2q-p}(z_{equi}(Y\times \A^q,r))(\Spec{k})\into N^{p+2\dim{Y}-2r}(Y\times\A^q)(q-\dim{Y}+r)\]
Specializing to $r=\dim{Y}$, this gives
\[C_{2q-p}(z_{equi}(Y\times \A^q,\dim{Y}))(\Spec{k})\into N^p(Y\times\A^q)(q)\]
Assuming resolution of singularities, Suslin also proved

\begin{thm}[\cite{S}, Thm.~3.2] If $k$ satisfies resolution of singularities then for any equidimensional quasiprojective scheme $Y$, the natural inclusion
\[C_{2q-p}(z_{equi}(Y\times\A^q,\dim{Y}))(\Spec{k})\into N^p(Y\times\A^q)(q)\]
is a quasi-isomorphism.
\end{thm}

The duality theorem of Friedlander and Voevodsky reads

\begin{thm}[\cite{FV}, Thm.~7.4] If $k$ satisfies resolution of singularities, $Y$ is a smooth quasiprojective $k$-scheme, and $X$ is a $k$-scheme if finite type, then for any $r$ the inclusion
\[z_{equi}(X,r)(Y)\into z_{equi}(Y\times X,r+\dim{Y})(\Spec{k})\]
induces a quasi-isomorphism
\[C_*\bigl(z_{equi}(X,r)\bigr)(Y)\xrightarrow{\sim} C_*\bigl(z_{equi}(Y\times X,r+\dim{Y})\bigr)(\Spec{k}).\]

\end{thm}




Combining the previous two theorems gives the desired quasi-isomorphism
\[\cZ^p(Y)(q)\xrightarrow{\sim}N^p(Y\times\A^q)(q).\]

\begin{rmk} Again, the above comparison relies on resolution of singularities. In fact, Friedlander and Suslin have shown (\cite{FS}, Corollary~12.2), for any perfect field, that the higher Chow groups are computed by hypercohomology with coefficients in the complexes of sheaves $\cZ^*(q)$. Moreover, (\cite{FV}, Theorem~8.1) show that the ``naive'' cohomology groups used above agree with the hypercohomology groups, but this latter result requires resolution of singularities.
\end{rmk}

\subsubsection{Products}

Now that we have functorial cycle complexes, we want to modify them to yield cdga's after tensoring with $\Q$. 

As above, the first step is to pass from simplicial complexes to cubical complexes. Again (\cite{L2}, Lemma~5.26.1), one can show that the cubical complexes $\cZ^*(Y)^c(*)$ are quasi-isomorphic to the simplicial ones. 

At this point, we have functorial dga's computing the higher Chow groups (assuming resolution of singularities). It only remains to make the dga's commutative. To clarify, the products we are considering are
\[\cZ^m(Y)^c(p)\otimes\cZ^n(Y)^c(q)\to\cZ^{m+n}(Y\times Y)^c(p+q)\xrightarrow{\Delta^*}\cZ^{m+n}(Y)^c(p+q).\]
Recall that $\cZ^m(Y)^c(p)$ is a subgroup of the group of cycles on $\square^{2p-m}\times Y\times\A^p$. There are thus two sources of noncommutativity: the box coordinate $m$ and the affine coordinate $p$. 
As above, we tensor our dga's with $\Q$ and consider the subspaces of cycles which are alternating with respect to permutation of the box coordinates. It can be shown (\cite{L2}, Lemma~5.26.2) that the result is quasi-isomorphic. 

However, this does not yet produce a cdga, since we have the extra affine coordinates. If we now restrict to the cycles which are {\em invariant} under permutation of the affine coordinates, then we finally obtain a cdga. Once again, this restriction does not change the quasi-isomorphism type (\cite{L2}, Lemma~5.27) (this essentially follows from homotopy invariance of the higher Chow groups). To sum up, we have the following presheaf of cdga's over $\Q$ which will serve as a model for the Bloch cycle complexes:


\begin{defn}\label{cAModel}The cdga $\cA^{*,*}(X)$ is defined by
\[\begin{split}
\cA^{p,q}(X)&=(\alt_{2q-p}(\cZ^p(X)^c(q)_\Q))^{\Sigma_q} \\
&=(\alt_{2q-p}(z_{equi}(\A^q,0)(X\times\square^{2q-p})\otimes\Q))^{\Sigma_q}.
\end{split}\]
That is, one takes the $\Q$-vector space generated by varieties $W\subseteq X\times\square^{2q-p}\times\A^q$ which are dominant and quasi-finite over $X\times\square^{2q-p}$. One then restricts to cycles which are alternating with respect to the $\Sigma_{2q-p}$ action on $\square^{2q-p}$. Finally one restricts to cycles which are invariant with respect to the $\Sigma_q$ action on $\A^q$.
\end{defn}

\begin{rmk} Note that the Beilinson-Soul\'{e} conjecture (\ref{BSConj}) for the field $k$ is precisely the statement that the dga $\cA(k)$ is cohomologically connected, meaning that $\rH^{p,q}(A)=0$ if $q=0$ and $p<0$ or $q>0$ and $p\leq 0$.
\end{rmk}

\subsection{Cocycles for functions}

It is known (\cite{NS}, \cite{T}) that
\[\rH^{n,n}(\Spec{k};\Z)\iso\CH^n(\Spec{k},n)\iso K^M_n(k),\]
where the groups on the right are the Milnor $K$-groups. In particular, in accord with \S\ref{ConjPicSec}, this says that 
\[\rH^{1,1}(\Spec{k};\Z)\iso\CH^1(\Spec{k},1)\iso k^\times,\]
the units in the field (this appeared already in \cite{B1}), and that $\CH^2(\Spec{k},2)$ is the quotient of $k^\times\otimes k^\times$ by the subgroup generated by Steinberg elements of the form $a\otimes(1-a)$, with $a\in k - \{0,1\}$. 

We will be interested in producing an explicit map 
\[\cO^\times(\Spec{k})=k^\times\to\cA^{1,1}(\Spec{k}).\]
In fact, it suffices to work with $\cZ^*(\Spec{k})^c(*)$ since we have an explicit projection $\cZ\otimes\Q\to\cA$.

The map $k^\times\to 
N^1(\Spec{k})^c(1)$ is defined by
\[a\mapsto C_a:=Z(t(a-1)-a)\subseteq\square^1.\]
In other words, $C_a$ is merely the point $\frac{a}{a-1}\in\square^1=\A^1$. To get a codimension $1$ cycle in $\square^1\times\A^1$, we pull back along the projection $\square^1\times\A^1\to\square^1$, yielding the cycle $C_a\times\A^1\subseteq\square^1\times\A^1$. Note, however, that $C_a\times\A^1\notin\cZ^1(k)^c(1)$ since this cycle is not quasifinite and dominant over $\square^1$.

\begin{rmk}When $a=1$, we have
\[C_a=Z(t(1-1)-1)=Z(-1)=\emptyset\subseteq\square^1,\]
so that $[C_1]=0\in\CH^1(\Spec{k})$.
\end{rmk}

Recall that the maps
\[\cZ^1(k)^c(1)\into N^1(\A^1)^c(1)\xleftarrow{\pi^*} N^1(k)^c(1)\]
are quasi-isomorphisms. We are looking for a cycle in $\cZ^1(k)^c(1)$ which maps to the class of $C_a\times\A^1$. 

\begin{defn}Denoting the coordinate on $\square^1$ by $x$ and the coordinate on $\A^1$ by $t$, the cycle
\[\Gamma_a:=Z(tx(x-1)-(x(a-1)-a))\subseteq\square^1\times\A^1\]
is quasi-finite and dominant over $\square^1$. This defines a map 
\[\Gamma_{(-)}:k^*\to\cZ^1(k)^c(1).\]
\end{defn}

The cycle $\Gamma_a$ is the graph of the rational function 
\[\varphi(x)=\frac{x(a-1)-a}{x(x-1)}=(a-1)\frac{x-\frac{a}{a-1}}{x(x-1)}.\] 
Moreover, we can see that it maps to the class of $C_a\times\A^1$, as an explicit homotopy $C_a\times\A^1\simeq\Gamma_a$ is given by
\[Z(tx(x-1)h-(x(1-a)-a))\subseteq\square^2\times\A^1\]
(see Figure~\ref{HtpyFig}). Here we are considering $x$ and $h$ as coordinates on $\square^2$. 

\begin{figure}
\ifthenelse{\PaperVersion=\NonDraftVersion}{
\begin{pspicture*}(-75pt,-3.5)(110pt,3.7)
\psset{xunit=45pt}
{\tiny\psaxes[Dx=1,Dy=5,dx=6,dy=6]
{<->}(0,0)(-1.6,-3.5)(2.4,3.7)
\psplot{-3}{-0.13}{x 1.4 sub x div x 1 sub div 2.5 div} 
\psplot{0.02}{0.99}{x 1.4 sub x div x 1 sub div 2.5 div}
\psplot{1.02}{2.7}{x 1.4 sub x div x 1 sub div 2.5 div}  
\psline[linewidth=.05, linestyle=dotted](1,-3.5)(1,3.5)
\psline[linewidth=.07]{->}(0.6,2)(0.6,3)
\psline[linewidth=.07]{->}(-0.7,-1.2)(-1.4,-2)
\psline[linewidth=.07]{->}(2,0.6)(1.7,2)
\psline[linewidth=.07]{->}(1.18,-1.3)(1.25,-2.6)
\uput{.3}[135](1,0){\small1}
\uput[-45](1.4, 0){\small$\frac{a}{a-1}$}
\psdot(1.4,0)
}
\end{pspicture*}\hspace{1.5em}
\begin{pspicture*}(-55pt,-3.5)(80pt,3.7)
\psset{xunit=35pt}
{\tiny\psaxes[dx=6,dy=6]
{<->}(0,0)(-1.6,-3.5)(2.2,3.7)
\psline[linewidth=.05](1.4,-3.5)(1.4,3.5)
\uput[-45](1.4, 0){\small$\frac{a}{a-1}$}
\psdot(1.4,0)
}
\end{pspicture*} }{} 
\caption{The cycle $\Gamma_a$ is deformed to $C_a\times\A^1$.}\label{HtpyFig}
\end{figure}
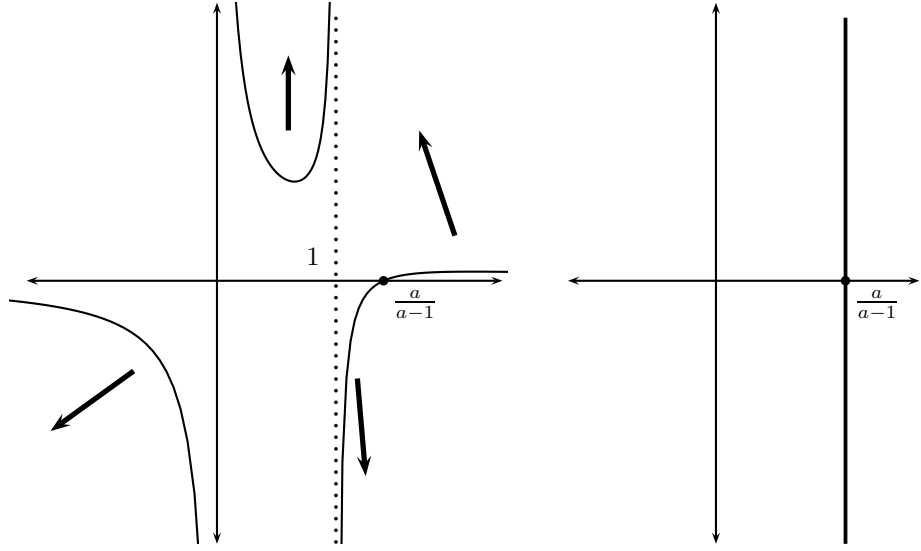


\vspace{3mm}

In fact, one has more generally an isomorphism
\[\rH^{1,1}(Y;\Z)\iso \cO^\times(Y).\]
\begin{defn}We define a map $\cO^\times(Y)\xrightarrow{\Gamma}\cZ^1(Y)^c(1)$ by
\[f\mapsto \Gamma_f:=Z\bigl(tx(x-1)-[x(f(y)-1)-f(y)]\bigr).\]
\end{defn}
As above one can show, by use of an explicit homotopy, that this is compatible with the map $\cO^\times(Y)\to N^1(Y)^c(1)$. We will hereafter write $[f]=\Gamma_f\in\cA^{1,1}(Y)$, as the notation is less cumbersome.

\begin{qst}
What are explicit cycles in $\cA^{1,2}(\Spec{k})$ realizing the Steinberg relations? This, and much more, was done in \cite{T} for the cubical Bloch complex.
\end{qst}

Although we do not know of a particular choice of bounding cochain for $[a]\cdot[1-a]$ in $\cA^{2,2}(k)$ we will use the notation $T_{a,1-a}$ to indiscriminately refer to some Totaro cycle bounding this product.



\section{Mixed Tate Motives}\label{MTMSec}

In this section, 
we review the theory of 
mixed Tate motives. 
In particular, we review the Bloch-Kriz and Kriz-May categories of mixed Tate motives.

\subsection{Mixed Tate categories}\label{MxdTtCatSubsec}

We begin by describing, via the formalism of Tannakian categories,
some of the formal structure expected of a category of mixed Tate motives.

\begin{defn} A {\bf mixed Tate category} is a neutral Tannakian $\Q$-linear category $\sM$ with an invertible object $\Q(1)$ such that 
\begin{enumerate}
\item any simple object is isomorphic to $\Q(m):=\Q(1)^{\otimes m}$ for some $m\in\Z$
\item 
\[\Hom_\sM(\Q(m),\Q(n))=\left\{\begin{array}{ll} \Q & m=n \\ 0 & m\neq n\end{array}\right.\]
\item\label{extproperty}
\[\Ext_\sM^1(\Q(m),\Q(n))=0\]
for $m\geq n$.
\end{enumerate}
\end{defn}

\begin{lemma} Each object $M\in\sM$ has a canonical finite increasing filtration $W_*M$ such that $Gr^W_{-2n}M\iso\oplus\Q(n)$ (and $Gr^W_{-2n+1}M=0$).
\end{lemma}


In fact, letting $\sM_{\leq 2n}$ be the full subcategory of objects $M$ such that $\Hom(\Q(i),M)=0$ for $i<-n$, the above defines a functor $W_{2n}:\sM\to\sM_{\leq 2n}$ which is right adjoint to the inclusion $\sM_{\leq 2n}\into \sM$. 

The weight filtration gives a canonical fiber functor $\sM\to{\mathrm Gr\!}\Vect_\Q$ by
\[M\mapsto \bigoplus_n \omega_n(M):= \bigoplus_n \Hom(\Q(-n),\gr^W_{2n}M).\]
Since the Tannaka group of ${\mathrm Gr\!}\Vect_Q$ is $\G_m$, 
the above, together with the inclusion ${\mathrm Gr\!}\Vect_Q\into\sM$ given by $V_n\mapsto V_n\otimes \Q(-n)$, gives a splitting
\[\G_m\into G_{\sM}\twoheadrightarrow \G_m,\]
where $G_{\sM}$ denotes the Tannaka group of $\sM$
Thus we have a decomposition $G_{\sM}\iso \G_m\ltimes\cU$ for some 
$\cU$. 

Furthermore, the $\G_m$-action on $\cU$ means that if $\cU\iso\Spec A$, then $A$ is a graded Hopf algebra. Using that $\gr^W_{2n+2}W_{2n}M=0$, one can show the following result:

\begin{lemma} $A_*$ is a connected (commutative) Hopf algebra, meaning that it is concentrated in non-negative degrees and $A_0=\Q$. 
\end{lemma}

Tannaka theory gives the following description of a mixed Tate category:

\begin{thm} The fiber functor $\omega_*:\sM\to{\mathrm Gr\!}\Vect_\Q$ induces an equivalence of $\otimes$-categories
\[\sM\simeq \Comod(A_*)\]
to the category of finite-dimensional (graded) $A_*$-comodules.
\end{thm}

This result shows that defining a mixed Tate category is equivalent to prescribing a Hopf algebra $A_*$. Below, we describe yet other means of defining such an $\sM$.

Let $M\in\sM$. Since $\cU$ preserves the weight filtration on $\omega(M)$ and acts trivially on $\omega(gr^W_*M)$ (since it acts trivially on $\Q(-n)$ by definition), $\cU$ acts unipotently on $\omega(M)$. Moreover, letting $\langle M\rangle\subseteq\sM$ be the full subcategory of subquotients of sums of copies of $M$,
we have that $\cU\mid_{\langle M\rangle}$ is unipotent. Since $\sM\iso\colim\langle M\rangle$, we deduce that $\cU$ is a pro-unipotent algebraic group.

Let $I=\ker(\epsilon:A_*\to\Q)$ be the augmentation ideal in $A_*$. One can form the ``indecomposable elements'' $QA_*:=I/I^2$. Define a map $\psi:I/I^2\to I/I^2\otimes I/I^2$ by $\psi=\Delta-\tau\circ\Delta$ where $\Delta:A\to A\otimes A$ is the comultiplication and $\tau:A\otimes A\iso A\otimes A$ is the transposition. This makes $I/I^2$ into a coLie algebra, in the sense that the induced structure on $L:=\Hom(I/I^2,\Q)$ is a Lie algebra. The Lie algebra $L$ is the usual Lie algebra of the algebraic group $\cU$. Note that it is concentrated in negative degrees. Recall also that the ``associated Lie algebra'' functor induces an equivalence of categories (over a field of characteristic $0$)
\[\left\{\mbox{unipotent algebraic groups} \right\} \leftrightarrow \left\{ \parbox{10em}{finite dimensional \\ nilpotent Lie algebras}\right\}\]
It follows that $L$ is a pro-nilpotent Lie algebra.

\begin{prop} There is an equivalence of categories $\Comod(A_*)\simeq \Rep(L)$, where in both cases we only allow finite dimensional objects. More generally, one can identify $A_*$-comodules with ``generalized nilpotent'' representations of $L$.
\end{prop}

\subsection{The approach of Bloch-Kriz}\label{BlchKrzSec}

As we have seen, to describe any mixed Tate category, it suffices to specify either a Hopf algebra or a Lie algebra. This is the approach followed in \cite{BK}, which we now review. The main idea, expounded already in \cite{B2}, is that one should consider the coLie algebra corresponding to the $1$-minimal model for (a connected version of) the Bloch cycle complex. This idea is improved upon in \cite{BK} in that the authors give a convenient construction of the $1$-minimal model. Their construction does {\it not} require the crucial Beilinson-Soul\'{e} vanishing conjecture. 

\begin{const} Let $R^*$ be an augmented cdga with augmentation ideal $I\subseteq R$. One can then form the two-sided bar construction $B(R):=B(\Q,R,\Q)$. This is a simplicial cochain complex over $\Q$ given in simplicial degree $n$ by the complex
\[ B_n(\Q,R,\Q)=\Q\otimes_\Q R^{\otimes n}\otimes_\Q \Q\iso R^{\otimes n}.\]
One may form an associated bicomplex in the usual way and totalize to produce a single complex. As usual when passing from simplicial objects to complexes, there are quasi-isomorphic reduced and unreduced versions. It will be more convenient to work with the reduced, or normalized, bar construction, which is given by
\[ \overline{B_n}(\Q,R,\Q)=I^{\otimes n}\]
Moreover, there is a graded commutative ``shuffle'' product as well as a coproduct defined on $B(R)$, making it into a commutative differential-graded Hopf algebra. It follows that
\[\chi_R:=\rH^0B(R)\]
is a commutative Hopf algebra.
\end{const}

We have been deliberately vague about some of the details of the above construction, but let us give a more careful treatment of $\chi_R$. First, note that an element of degree $0$ of $\overline{B}(R)$ is an element $r_1\otimes\dots\otimes r_n$ such that $|r_1|+\dots+|r_n|=n$. In particular, if $R$ is connected, this implies that each $r_i$ must have degree $1$. The differential is given by
\[\begin{split}
 D(r_1\otimes\dots\otimes r_n) &=\sum_{i=1}^n (-1)^{r_1+\dots+r_{i-1}}r_1\otimes\dots\otimes d(r_i)\otimes\dots\otimes r_n \\
 & \qquad \qquad +\sum_{j=1}^{n-1} (-1)^{n+j}r_1\otimes\dots\otimes r_jr_{j+1}\otimes\dots\otimes r_n.
 \end{split}\]



Recall that if $R$ is an augmented algebra with augmentation ideal $I\subseteq R$, the set of {\bf indecomposables} in $R$ is $QR:=I/I^2$.

\begin{defn} Let $R_*$ be a cohomologically connected cdga. Sullivan observed that the structure of a coLie algebra on a vector space $V$ is equivalent to the structure of a dga on $\Lambda(V[-1])$. A
 {\bf $1$-minimal model} for $R_*$ is a coLie algebra $\gamma$ together with a map of dga's $\Lambda(\gamma[-1])\to R_*$ inducing an isomorphism on $\rH^1$ and a monomorphism on $\rH^2$. It is known that any cohomologically connected dga admits a unique $1$-minimal model. We say that $R_*$ is a $K(\pi,1)$ if it is quasi-isomorphic to its $1$-minimal model.
\end{defn}

\begin{thm}[\cite{BK}, Theorem~2.30] Let $R$ be a cohomologically connected cdga. Then the Hopf algebra $\chi_R$ is a free commutative algebra. 
Moreover, $\Lambda Q\chi_R$ is a $1$-minimal model for $R$.
\end{thm}

Recall, as stated above, that the statement that $\cA(\Spec{k})$ is cohomologically connected is precisely the Beilinson-Soul\'{e} conjecture (\ref{BSConj}). Bloch-Kriz thus define the category $MTM_{BK}$ of mixed Tate motives to be the category of (graded) comodules over the above (graded) Hopf algebra. Equivalently, mixed Tate motives are representations of the (graded) coLie algebra $Q\chi_\cA$. Note that if $\Q(r)$ denotes the trivial representation concentrated in degree $n$, then one has
\[\Ext_{MTM_{BK}}^i(\Q,\Q(r))\iso \rH^i(Q\chi_\cA)_{\mathrm{degree}\ r}\iso\rH^{i,r}(\Lambda Q\chi_R).\]
The following result is an immediate consequence.

\begin{prop}
If $\cA$ is a $K(\pi,1)$ then one has
\[\Ext_{MTM_{BK}}^i(\Q,\Q(r))\iso \rH^{i,r}(k;\Q).\]
\end{prop}

One of the key new features of the article \cite{BK} is the description of \'{e}tale as well as Hodge realizations functors, though we do not say more about that here.

\subsubsection{
Review of cofibers}\label{cofibsec}

In section~\ref{KMSect}, we will outline the approach to mixed Tate motives of Kriz and May. This approach begins with the derived category $\sD(\cA)$ of cohomologically bounded below $\cA$-modules. In the next two sections,
we review some relevant homological algebra.

In this section, we will reproduce some of the discussion from section~III.1 of \cite{KM}, including sign conventions. 
$R^*$ will denote a 
commutative dga.

Recall the unit interval $R$-module $I$. It is free of rank one in degree $-1$, with generator denoted $[I]$, and free of rank two in degree $0$, with generators denoted $[0]$ and $[1]$. The differential is determined by $d[I]=[0]-[1]$. The unit interval $I$ is introduced in order to define homotopies: given dg $R$-modules $M$ and $N$ and two maps $f,g:M\to N$, a homotopy from $f$ to $g$ is a map $h:M\otimes_R I\to N$ such that $h(m\otimes[0])=f(m)$ and $h(m\otimes[1])=g(m)$. Note that since we have chosen to write our unit interval on the right of $M$, we have the formula 
\[d(m\otimes[I])=dm\otimes[I]+(-1)^{\deg{m}}m\otimes([0]-[1]).\]
This gives the formula
\[d(h(m\otimes[I]))=h(dm\otimes[I])+(-1)^{\deg{m}}(f(m)-g(m)).\]

We define the {\bf cone} on $M$ to be the pushout

\centerline{\xymatrix{
M\otimes_R R[1] \ar[r] \ar[d] & 0 \ar[d] \\
M\otimes_R I \ar[r] & CM.}}

\noindent The induced differential is $d(m\otimes[I])=dm\otimes[I]+(-1)^{\deg{m}}m\otimes[0]$. Note that there is a canonical map $M\iso M\otimes_R R[0]\into CM$. More generally, given any map $f:M\to N$ of dg $R$-modules, we define the {\bf cofiber} of $f$ by the pushout diagram

\centerline{\xymatrix{
M \ar[r]^f \ar[d] & N \ar[d] \\
CM \ar[r] & C(f).}}

\noindent A typical element of $C(f)$ can be written in the form $n+m\otimes[I]$, and the induced differential is 
\[d(n+m\otimes[I])=dn+(-1)^{\deg{m}}f(m)+dm\otimes[I].\]
It is often useful to write the differential in the form of the matrix
\[\left(\begin{array}{cc} d_M & 0 \\ \overline{f} & d_N\end{array}\right),\]
where $\overline{f}(x)=(-1)^{|x|}f(x)$.
Given the above definitions, it is now straightforward to verify

\begin{lemma}\label{cofiber}Suppose given dg $R$-modules and maps between them as shown in the diagram

\centerline{\xymatrix{
W \ar[r]^\alpha \ar[d]_f \ar@{-->}[dr]^h & Y \ar[d]^g \\
X \ar[r]^\beta & Z,}}

\noindent where $h$ is a homotopy $h:\beta f\simeq g\alpha$. Then the assignment
\[x+w\otimes[I]\mapsto \beta(x)+h(w)+\alpha(w)\otimes[I]\]
defines a map of dg $R$-modules $C(f)\to C(g)$.
\end{lemma}

\subsection{The homotopy category of cell modules}

In this section, we review the description of the derived category of a commutative dga as the homotopy category of so-called ``cell modules''.

The notion of a cell module is based on that of a cell complex in topology. Recall that a CW complex is a topological space that is built up inductively by attaching cells of higher and higher dimension. More generally, there is the notion of cell complex in which one drops the condition on the dimension of the cells. In order to mimic this definition in a category of modules over a dga $R^*$, one needs only know what should be the appropriate notions of sphere and disk modules. All modules will be taken to be left $R$-modules.

\begin{defn} Let $R^*$ be a dga. The {\bf $n$-sphere $R$-module} $S_R^n$ is defined to be the free $R$-module with generator $i^n$ in degree $n$. Similarly, the {\bf $n$-disk $R$-module} $D_R^n$ is defined to be the cone $D_R^n=CS_R^{n+1}$ on the $(n+1)$-sphere. 
\end{defn}

With the above notions, we are now ready to describe cell modules:

\begin{defn}\label{CellMod} Let $R^*$ be a dga. A {\bf cell $R$-module} is an $R$-module $M$ of the form $M\iso\bigcup_n M_n$, where $M_0=0$ and $M_n$ is obtained from $M_{n-1}$ by a pushout diagram

\centerline{ \xymatrix {
\bigoplus_i S^{n_i}_R \ar[r] \ar@{^{(}->}[d] & M_{n-1} \ar@{^{(}->}[d] \\
\bigoplus_i D^{n_i-1}_R \ar[r] & M_n.
}}
\end{defn}

Note that if one ignores the differential, a cell $R$-module is a free module. Thus, cell modules are sometimes called ``semi-free'' modules.

The dga of interest for us will be the Bloch cycle complex, or rather the variant $\cA$ (\ref{cAModel}). Recall that this dga has a secondary grading. We will thus need a notion of cell module over dgas with an additional grading.

\begin{defn}An {\bf Adams-grading} on a dga $R^*$ is a second grading such that $R^{p,q}=0$ for $q<0$ or $p>2q$. The differential is of degree $0$ with respect to the Adams grading and the product preserves both gradings. In the case of a commutative dga, the Adams grading will {\em not} contribute a sign to commutativity.
%

An {\bf Adams-graded module} over an Adams-graded dga $R$ is a dg-$R$-module $M$ equipped with a second grading which is compatible with the Adams-grading of $R$. We say $M$ is {\bf cohomologically bounded below} if $\rH^{p,*}(M)=0$ for $p\ll 0$.
\end{defn}

One has as above sphere $S^{p,q}_R$ and disk $D^{p,q}_R$ modules for all $p,q\in\Z$. An Adams-graded cell module is thus built out of Adams-graded cells. If $R$ is Adams-graded, the unit interval $I$ of \S\ref{cofibsec} is defined to have generators in Adams grading $0$.

\begin{defn} Let $R^*$ be an Adams-graded dga. The {\bf derived category} $\sD(R)$ of cohomologically bounded below $R$-modules is obtained from the category of cohomologically bounded below dg-$R$-modules $M$ by formally inverting the quasi-isomorphisms.
\end{defn}

The above definition is only reasonable when $R$ is itself cohomologically bounded below. Our main example will be the cdga $\cA$ (\ref{cAModel}); again, the Beilinson-Soul\'{e} conjecture (\ref{BSConj}) says that this dga is cohomologically connected.

The point of introducing cell modules is that they provide a more convenient description of $\sD(R)$. Let $\Mod_R$ denote the category of (cohomologically bounded below) $R$-modules and $\Cell_R$ denote the full subcategory of (cohomologically bounded below) cell $R$-modules. One has a corresponding full embedding $h\!\Cell_R\into h\!\Mod_R$ of homotopy categories, obtained by passage to homotopy classes of morphisms. Since homotopy equivalences are quasi-isomorphisms, one has by the universal property of localization a functor $h\!\Mod_R\to\sD(R)$.

\begin{thm}\label{CellEqDer} The composite $h\Cell_R\into h\Mod_R\to\sD(R)$ is an equivalence of categories.
\end{thm}

The main results needed to prove this are 

\begin{thm}[Whitehead, \cite{KM} Thm.~III.2.3] If $M$ is a cell $R$-module and $e:N\to P$ is a quasi-isomorphism of $R$-modules then $e_*:h\!\Mod_R(M,N)\to h\!\Mod_R(M,P)$ is an isomorphism.
\end{thm}

\noindent and

\begin{thm}[Cellular approximation, \cite{KM} Thm.~III.2.6] For any $R$-module $M$, there is a cell $R$-module $N$ and a quasi-isomorphism $e:N\xrightarrow{\sim} M$.
\end{thm}

The following results will also be useful.

\begin{prop}[\cite{KM}, Lemma~III.4.1] Let $R$ be a cdga and $M$ be a cell-$R$-module. Then $M\otimes_R-:\Mod_R\to\Mod_R$ preserves quasi-isomorphisms. In particular, if $M$ is a cell $R$-module and $N$ is any $R$-module, we have $M\otimes^{\bL}_R N\iso M\otimes_R N$ in $\sD(R)$.
\end{prop}

\begin{prop}[\cite{KM}, Proposition~III.4.2]\label{DerivedInvarProp} Let $\varphi:R\to R'$ be a quasi-isomorphism of dga's. Then pullback functor $\varphi^*:\sD(R')\to\sD(R)$ is an equivalence of categories.
\end{prop}

\subsection{The approach of Kriz-May}\label{KMSect}


We recall from \cite{BBD} the notion of a ``t-structure''.

\begin{defn} Let $\cT$ be a triangulated category. A {\bf t-structure} on $\cT$ is a pair of full subcategories $\cT^{\leq 0}$ and $\cT^{\geq 0}$ such that
\begin{enumerate}
\item $\cT^{\leq 0}\subset \cT^{\leq 1}:=\cT[-1]$ and $\cT^{\geq 1}:=\cT^{\geq 0}[-1]\subset \cT^{\geq 0}$
\item $\Hom(X,Y)=0$ if $X\in\cT^{\leq 0}$ and $Y\in\cT^{\geq 1}$
\item For any $X\in\cT$ there is a distinguished triangle $X^{\leq 0}\to X\to X^{\geq 1}\to X^{\leq 0}[1]$ with $X^{\leq 0}\in\cT^{\leq 0}$ and $X^{\geq 1}\in\cT^{\geq 1}$.
\end{enumerate}
One refers to $\cT^0:=\cT^{\leq 0}\cap\cT^{\geq 0}$ as the {\bf heart} of the t-structure.
\end{defn}

\begin{thm}[\cite{BBD}]\label{HeartThm} The heart of a t-structure is an abelian category.
\end{thm}

Since the heart $\cT^0$ is an abelian category, one can form its derived category $\sD(\cT^0)$, but this need not agree with $\cT$. In fact, there is no canonical functor in general. Moreover, even in cases where one has a functor $\sD(\cT^0)\to\cT$, one has that the induced map on $\Ext$ groups is an isomorphism in degree $0$ and $1$ but only an injection in degree $2$. This holds since the $\Ext$ groups in $\sD(\cT^0)$ are necessarily generated by $\Ext^1$, whereas this need not be the case in $\cT$.

Kriz and May consider the triangulated category $\sD(A)$ of cohomologically bounded below $A$-modules. They define subcategories $\sD^{\leq 0}\subset \sD(A)$ and $\sD^{\geq 0}\subset \sD(A)$ by
\[\sD^{\leq 0}=\{M\in\sD(A) \mid \rH^n(M\otimes^{\bL}_A \Q)=0, n<0\}\]
and
\[\sD^{\geq 0}=\{M\in\sD(A) \mid \rH^n(M\otimes^{\bL}_A \Q)=0, n>0\}.\]
Let $\cH:=\sD^{\leq 0}\cap\sD^{\geq0}$  and $\cF\cH\subseteq\cH$ be the full subcategory of $M\in\cH$ such that $\rH^0(M\otimes^{\bL}_A \Q)$ is finite dimensional.

\begin{thm}[\cite{KM}, Theorem~IV.1.1] If the Beilinson-Soul\'{e} conjecture holds for $k$, then the above defines a $t$-structure on $\sD(A)$. Moreover, the functor $\omega:\cF\cH\to\Vect_\Q$ defined by $\omega(M)=\rH^0(M\otimes^{\bL}_A\Q)$ is a fiber functor, making $\cF\cH$ into a Tannakian category.
\end{thm}

Again, if we take $h\!\Cell_R$ as a model for the derived category $\sD(R)$, then the ordinary tensor product coincides with the derived tensor product.

\begin{defn}The Kriz-May category of mixed Tate motives $MTM_{KM}$ is the category $\cF\cH$.
\end{defn}

\begin{thm}[\cite{KM}, Theorem~IV.1.10] If the Beilinson-Soul\'{e} conjecture holds for $k$, the Kriz-May category of mixed Tate motives $MTM_{KM}$ is equivalent to the Bloch-Kriz category $MTM_{BK}$ (\S\ref{BlchKrzSec}).
\end{thm}

In the category $MTM_{KM}$, the role of the objects $\Q(j)$ of \S\ref{MxdTtCatSubsec} is played by the cell modules $A(r)$, where the notation $A(r)$ signifies the free $A$ module on a generator in degree $(0,-r)$.

As we stated above, there is no guarantee that the $\Ext$ groups computed in the derived category of the heart of a $t$-structure on a triangulated category agree with the $\Ext$ groups in the triangulated category. Translated into this context, this means that $\Ext$ groups between mixed Tate motives do not necessarily compute motivic cohomology. The following result gives a sufficient condition.

\begin{thm}[\cite{KM}, Theorem~IV.1.11] If the motivic dga $\cA$ is a $K(\pi,1)$ then $\Ext^i_{MTM}(\Q,\Q(j))\iso\rH^{i,j}(k;\Q)$.
\end{thm}

\subsection{Minimal modules}\label{MinModSect}

Now let $R^*$ be a connected dga, and let as usual $I\subseteq R$ be the augmentation ideal.

\begin{defn} We say a bounded below cell $R$-module is {\bf minimal} if it has decomposable differential, meaning that $d(M)\subset (IR)M$.
\end{defn}

Kriz and May prove the following useful results about minimal modules. 

\begin{prop}[\cite{KM}, Proposition~IV.3.3] A bounded below cell $R$-module is minimal if and only if $d=0$ on $M\otimes_R\Q$. Moreover, a quasi-isomorphism $f:M\to N$ between minimal $R$-modules is necessarily an isomorphism.
\end{prop}

\begin{prop}[\cite{KM}, Theorem~IV.3.7] Given any $R$-module $N$, there is a minimal $R$-module $M$ and a quasi-isomorphism $M\xrightarrow{\sim}N$. Moreover, $M$ is unique up to isomorphism.
\end{prop}

This last result has as consequence the following

\begin{cor} Let $\mathrm{Min}_R$ denote the category of minimal $R$-modules. The composite
\[ 
h\mathrm{Min}_R\into h\Cell_R\xrightarrow{\sim} \sD(R)\]
is an equivalence of categories.
\end{cor}

It is this result which allows Kriz and May to establish a comparison of $MTM_{KM}$ with $MTM_{BK}$. Indeed, given a cell module $N$ in $MTM_{KM}$, one can replace it by a minimal module $M$. The indecomposables $M\otimes_R\Q$ give then the desired object of $MTM_{BK}$.

\section{The motivic fundamental group}\label{CellSec}

We write simply $\cA$ for $\cA(\Spec k)$, where $\cA(X)$ is the functorial motivic dga of Definition~\ref{cAModel}. Our goal is to describe a cellular approximation to the motivic version
\begin{equation}
\cA^*(X^n)\to \cA^*(a\times X^{n-1})\oplus 
\dots\oplus \cA^*(X^{n-1}\times b)\to\dots\to \cA^*(pt)^{\oplus^{n+1}}
\label{theMotivicComplex}\end{equation}
of the complex~(\ref{theComplex}) in the case $X=\bP^1-\{0,1,\infty\}$.


\subsection{The $n=2$ case}\label{n2case}



Let $X=\bP^1-\{0,1,\infty\}$, and let $a\neq b\in X(k)$ be rational points. 
When $n=2$, the complex~(\ref{theMotivicComplex}) becomes
\begin{equation}\begin{split}
\cA(X\times X) & \xrightarrow{\beta} \cA\bigl(\{a\}\times X\bigr)\oplus \cA(\Delta_X)\oplus \cA\bigl(X\times\{b\}\bigr) \\
 & \quad \xrightarrow{\beta'} \cA(\{a,a\})\oplus \cA(\{a,b\})\oplus \cA(\{b,b\}).
\end{split}\notag\end{equation}
Here the first map is given by the restriction along the inclusions, where restriction onto the diagonal is taken with a minus sign. The second map is given by restriction to $a$ minus restriction to $b$. More explicitly, the map $\beta'$ is defined by
\[\beta'(x,y,z)=(x_{\mid a}-y_{\mid a},-x_{\mid b}+z_{\mid a},-y_{\mid b}-z_{\mid b}).\]
The above is a complex of dg-\cA-modules, so that the associated total complex defines a dg-\cA-module. We will describe a cellular approximation to this total complex.



We will write $\cA\{n\}$ for $\cA(n)[n]$, the $\cA$-module obtained by shifting both the Adams and homological gradings by $n$; that is, $\cA\{n\}^{p,q}=\cA^{p+n,q+n}$. Recall that there is a quasi-isomorphism $\cA\oplus \cA\{-1\}\simeq \cA(\G_m)$. If $\pi:\G_m\to\Spec(k)$ is the structure morphism and if $t$ is the standard parameter on $\G_m$, this map is given by
\[(x,y)\mapsto \pi^*(x)+\pi^*(y)\cdot [t].\]
Together with the Mayer-Vietoris sequence associated to the covering 
\[(\A^1-\{0\})\cup(\A^1-\{1\})=\A^1\] 
(note that $(\A^1-\{0\})\cap(\A^1-\{1\})=X$), this gives a quasi-isomorphism 
\[\cA\oplus \cA\{-1\}\oplus \cA\{-1\}\xrightarrow{\sim} \cA(X).\]
Explicitly, this is
\[(x,y,z)\mapsto \pi^*(x)+\pi^*(y)\cdot [t]+\pi^*(z)\cdot [1-t].\]
Similarly, one gets a quasi-isomorphism
\[\cA\oplus \cA\{-1\}^4\oplus \cA\{-2\}^4\xrightarrow{\sim} \cA(X\times X)\]
given explicitly by
\begin{equation}
\begin{split}
(x,y_1,\dots,y_4,z_1,\dots,z_4) &\mapsto \pi^*x+\pi^*y_1\cdot [U]+\pi^*y_2\cdot[1-U] \\
 & \qquad +\pi^*y_3\cdot [V]+\pi^*y_4\cdot[1-V] \\
 & \qquad +\pi^*z_1\cdot [U]\cdot[V]+\pi^* z_2\cdot [1-U]\cdot[V] \\
 & \qquad +\pi^* z_3\cdot [U]\cdot[1-V]+\pi^*z_4\cdot[1-U]\cdot[1-V],
\end{split}\notag\end{equation}
where $U$ is the parameter on the first factor and $V$ is the parameter on the second factor. Keeping the above quasi-isomorphisms in mind, we will label the generators of the various summands by $[U]$, $[1-U]\cdot[V]$, etc.

Using these quasi-isomorphic free $\cA$-modules, we attempt to build a complex
\[\cA\oplus \cA\{-1\}^4\oplus \cA\{-2\}^4\xrightarrow{\alpha}\cA^3\oplus \cA\{-1\}^6\xrightarrow{\alpha'}\cA^3.\]
There is no trouble in defining $\alpha'$; this is defined to be restriction to $a$ minus restriction to $b$ as above. Similarly, there is no trouble defining $\alpha$ on the $\cA$ and $\cA\{-1\}$ summands. The difficulty comes when trying to define $\alpha$ on the $\cA\{-2\}$ summands. Writing $W$ for the parameter on the diagonal $\Delta_X$, one wants to set 
\[\alpha([U]\cdot[1-V])=[a]\cdot [1-V]- [W]\cdot [1-W]+ [U]\cdot [1-b],\]
but of course we can't since $[W]$ is not an element of $(\cA\{-1\})^2(2)=\cA^1(1)$. Instead, we simply define
\[\alpha([U]\cdot[1-V])=[a]\cdot [1-V]- [1-b]\cdot [U]\]
 (we have switched the order of $[U]$ and $[1-b]$ here, introducing a sign, because $[U]$ denotes the generator of the left $\cA$-module $\cA\{-1\}$). We define $\alpha$ similarly on the other $\cA\{-2\}$ summands. 

Unfortunately, the above definition of $\alpha$ does not yield a complex. As expected, we have $\alpha'\circ\alpha=0$ on the $\cA$ and $\cA\{-1\}$ summands. On the $\cA\{-2\}$ summands we have
\[\alpha'\circ\alpha([U]\cdot[V])=([a]^2,0,-[b]^2)\]
and
\[\alpha'\circ\alpha([1-U]\cdot[1-V])=([1-a]^2,0,-[1-b]^2).\]
These both vanish by graded-commutativity since the elements have (homological) grading $1$, and we are working over $\Q$. On the other hand,
\[\alpha'\circ\alpha([1-U]\cdot[V])=([1-a]\cdot[a],0,-[1-b]\cdot[b])\]
and
\[\alpha'\circ\alpha([U]\cdot[1-V])=([a]\cdot[1-a],0,-[b]\cdot[1-b]).\]
These do not vanish, but they are coboundaries since the Steinberg relation holds in motivic cohomology (\ref{SteinbergRltn}).
Choices of bounding cycles allow one to define a contracting homotopy $h:\alpha'\circ\alpha\simeq0$.


Let us suppose for a moment that we have chosen appropriate bounding cycles and defined a contracting homotopy. Using Lemma~\ref{cofiber}, this homotopy allows us to build a cell module out of our homotopy-complex. We must now compare this cell module to the original complex. At this point, we have the (not commutative!) solid arrow diagram

\centerline{\xymatrix @R=35pt {
\cA_{X\times X} \ar[r]^{\beta} & \cA_{a\times X}\oplus \cA_{\Delta_X}\oplus \cA_{X\times b} \ar[r]^(0.55){\beta'} \ar[r] & \cA_{a,a}\oplus \cA_{a,b}\oplus \cA_{b,b} \\
\cA\oplus \cA\{-1\}^4\oplus \cA\{-2\}^4 \ar[r]^{\alpha} \ar[u]_f \ar@{-->}[ur]^H \ar@/_2pc/@{-->}[rr]^h & \cA^3\oplus \cA\{-1\}^6 \ar[r]^(0.55){\alpha'} \ar[u]_{f'} & \cA^3 \ar@{=}[u],
}}

\vspace{3mm}\noindent where, for instance, $\cA_{a,a}$ denotes $\cA(\{a,a\})\iso\cA$. The right-hand square commutes, but the $\cA\{-2\}$ summands cause the left-hand square to only commute up to homotopy. The composites $\beta\circ f$ and $f'\circ \alpha$ differ by the element $[W]\cdot[1-W]$ on the summand $\cA\{-2\}[U]\cdot[1-V]$ and by $[1-W]\cdot[W]$ on the summand $\cA\{-2\}[1-U]\cdot[V]$. Defining a commuting homotopy $H:f'\circ\alpha\simeq\beta\circ f$ thus amounts to choosing a bounding cycle $T_{W,1-W}$ for $[W]\cdot[1-W]$ and for $[1-W]\cdot[W]$.

Suppose that we have chosen such a homotopy $H$. To compare our cell module to the original complex, it suffices to find a homotopy $\Theta:h\simeq \beta'\circ H$. But note that given any choice of $H$, we can simply define $h$ to be $\beta'\circ H$. Thus $\Theta$ is not necessary; the only required data is the homotopy $H$.

Any choice of $H$ now produces  a cell module, together with a quasi-isomorphism to the original complex. Thus any two choices will produce homotopy equivalent cell modules. 

\vspace{3mm}

We have given a construction of the truncated motivic fundamental group of \mbox{$X=\bP^1-\{0,1,\infty\}$} in the case $n=2$, using only the Steinberg relation $[t]\cdot[1-t]=0$ in $\rH^{2,2}(X)$ (this is the universal case). We will generalize this to the $n=3$ case in \S\ref{n3case}, and we give the general result in \S\ref{gensec}.

\subsection{Massey products}\label{Masseysec}

We now briefly discuss Massey products (see, e.g., \cite{M}), as these will be needed below.

Let $R^*$ be a dga. Given a cochain $r\in R^n$, we let $\overline{r}$ denote $(-1)^{n-1}r$. This will be convenient later.

Suppose given $\alpha_1,\alpha_2,\alpha_3\in\rH^*(R)$ such that $\alpha_1\alpha_2=0$ and $\alpha_2\alpha_3=0$. Choosing representing cocycles $a_1,a_2,a_3\in R^*$, so that $[a_i]=\alpha_i$, this means that there are cochains $a_{12},a_{23}\in R^*$ such that 
\[d(a_{12})=a_1a_2 \qquad \text{ and } \qquad d(a_{23})=a_2a_3.\]
Actually, it is more convenient to choose $a_{12}$ and $a_{23}$ so that $d(a_{12})=\overline{a_1}a_2$ and $d(a_{23})=\overline{a_2}a_3$. The cochain
\[\overline{a_{12}}a_3+\overline{a_1}a_{23}\]
is then a cocycle. We denote the cohomology class it represents by $\langle \alpha_1, \alpha_2, \alpha_3\rangle$ and refer to it as the {\bf triple Massey product}.

Unfortunately, the above class is not determined by $\alpha_1$, $\alpha_2$, and $\alpha_3$. The cohomology class is sensitive to the choices of $a_{12}$ and $a_{23}$; adding a cocycle to either of these cochains yields a new cohomology class which has equal right to be called the triple Massey product. Thus the triple Massey product $\langle\alpha_1,\alpha_2,\alpha_3\rangle$ is not an element of $\rH^*(R)$ but rather a coset of $\alpha_1\rH^*(R)+\rH^*(R)\alpha_3$ in $\rH^*(R)$. The subgroup $\alpha_1\rH^*(R)+\rH^*(R)\alpha_3\subseteq\rH^*(R)$ is called the {\bf indeterminacy} of this Massey product.

We will need more generally the notion of an $n$-fold Massey product, which is defined inductively. Suppose given $\alpha_1,\dots,\alpha_n\in \rH^*(R)$ with choices of representing cocycles $a_i\in R^*$, $[a_i]=\alpha_i$. Suppose moreover that we have defined the notion of $3$-fold, $4$-fold, \dots, and $n-1$-fold Massey products. Finally, suppose that all of the Massey products $\langle \alpha_i,\dots,\alpha_j\rangle$ are {\em compatibly} defined and contain zero for $j-i<n-1$. That these Massey products contain zero means that we have bounding cochains $a_{i,\dots,j}$ for $\langle\alpha_i,\dots,\alpha_j\rangle$. The compatibility condition is as follows: a choice of cochain $a_{23}$ appears in both $d(a_{13})$ and $d(a_{24})$; we require both choices to be the same. Then the cochain
\[\sum_{i=1}^{n-1}\overline{a_{1,i}}a_{i+1,n}\]
is a cocyle, and we denote the resulting cohomology class by $\langle \alpha_1,\dots,\alpha_n\rangle$. Again, there is (more complicated) indeterminacy involved. 

The Massey products that will arise below will be iterated Massey products of the elements $[t]$ and $[1-t]$ in 
$\rH^{*,*}(X)$.
In fact, we will see that 
a sufficient condition for building a cellular approximation to the complex~(\ref{theMotivicComplex}) will be that these iterated Massey products are defined and contain zero. As we have already used the notation $T_{t,1-t}$ for a cochain bounding the Massey product $\langle [t],[1-t]\rangle$, we will use the notation $T_{x_1,\dots,x_n}$ for a cochain bounding the Massey product $\langle x_1,\dots,x_n\rangle$, and we will refer to these as generalized Totaro cycles (do not confuse cycle with cocyle here!).

\subsection{The $n=3$ case}\label{n3case}

We now look at the next simplest case, when $n=3$. We are interested in obtaining a cellular approximation to the complex

\begin{equation}\begin{split}\cA(X^3)&\xrightarrow{\beta_1}\cA(\{a\}\times X^2)\oplus\cA(\Delta_X\times X)\oplus\cA(X\times\Delta_X)\oplus\cA(X^2\times\{b\}) \\
&\xrightarrow{\beta_2}\cA(\{a,a\}\times X)\oplus\cA(\{a\}\times \Delta_X)\oplus\cA(\Delta_{X^3}) \\
&\qquad \oplus \cA(\Delta_X\times\{b\})\oplus\cA(X\times\{b,b\})\oplus\cA(\{a\}\times X\times\{b\}) \\
&\xrightarrow{\beta_3}\cA(a,a,a)\oplus\cA(a,a,b)\oplus\cA(a,b,b)\oplus\cA(b,b,b)
\end{split}\notag\end{equation}
A Mayer-Vietoris argument again gives us a cellular approximation
\[\cA\oplus\cA\{-1\}^6\oplus\cA\{-2\}^{12}\oplus\cA\{-3\}^8\xrightarrow{\sim}\cA(X^3).\]

A cellular approximation to the above complex will then consist of maps as in the diagram

\begin{equation}
\xymatrix{
\cA\oplus\cA\{-1\}^6\oplus\cA\{-2\}^{12}\oplus\cA\{-3\}^8 \ar[rr]^(.65){f_3}_(.65){\sim} \ar[d]^{\alpha_3} 
& & \cA(X^3) \ar[d]^{\beta_3} \\
\cA^4\oplus\cA\{-1\}^{16}\oplus\cA\{-2\}^{16} \ar[rr]^(.65){f_2}_(.65){\sim} \ar[d]^{\alpha_2} 
& & \cA(X^2)^4 \ar[d]^{\beta_2} \\
\cA^6 \oplus\cA\{-1\}^{12} \ar[rr]^(.65){f_1}_(.65){\sim} \ar[d]^{\alpha_1} & & \cA(X)^6 \ar[d]^{\beta_1} \\
\cA^4 \ar@{=}[rr] & & \cA^4, 
}\label{n3diag}
\end{equation}

\noindent together with appropriate homotopies. Repeated application of Lemma~\ref{cofiber}, together with the fact that the bottom rectangle strictly commutes, gives that the required homotopies are
\begin{align*}
h_3^2&:0\simeq\alpha_2\alpha_3, \\
h_2^2&:0\simeq\alpha_1\alpha_2, \\
h_3^3&:\alpha_1h_3^2-h_2^2\alpha_3\simeq0, \\
H_3^1&:f_2\alpha_3\simeq\beta_3 f_3, \\
H_2^1&:f_1\alpha_2\simeq\beta_2 f_2, \\
H_3^2&:f_1h_3^2+H_2^1\alpha_3+\beta_2H_3^1\simeq0, \\
H_2^2&:\beta_1 H_2^1+h_2^2\simeq0, 
\end{align*}
and
\[H_3^3:h_3^3+H_2^2\alpha_3-\beta_1H_3^2\simeq0.\]

As we remarked in \S\ref{cofibsec}, it is often convenient to write the differentials on cofibers in matrix form. Here the matrices for the two cofibers would be
\[
C=\left(\begin{array}{cccc} d & 0 & 0 & 0 \\ \overline{\alpha_3} & d & 0 & 0 \\ \overline{h_3^2} & \overline{\alpha_2} & d & 0 \\ \overline{h_3^3} & \overline{h_2^2} & \overline{\alpha_1} & d \end{array}\right) \qquad \text{and} \qquad
B=\left(\begin{array}{cccc} d & 0 & 0 & 0 \\ \overline{\beta_3} & d & 0 & 0 \\ 0 & \overline{\beta_2} & d & 0 \\ 0 & 0 & \overline{\beta_1} & d \end{array}\right)
\]
The desired quasi-isomorphism of cofibers would be a matrix
\[H=\left(\begin{array}{cccc} f_3 & 0 & 0 & 0 \\ H_3^1 & f_2 & 0 & 0 \\ H_3^2 & H_2^1 & f_1 & 0 \\ H_3^3 & H_2^2 & H_1^1 & \mathrm{id}\end{array}\right)\]
such that $HC=BH$ (again, $H_1^1=0$).

Replacing the modules on the left column of (\ref{n3diag}) by $C_3$, \dots, $C_0$ and those on the right column by $B_3$, \dots, $B_0$, the desired homotopies are represented in the diagram

\centerline{\xymatrix @C=70pt @R=35pt {
C_3 \ar[r]^{f_3} \ar[d]_{\alpha_3} \ar@{==>}@/_4ex/[dd]_{h_3^2} \ar@{-->}@/_9ex/[ddd]_{h_3^3} \ar@{==>}[dr]_(.59){H_3^1} \ar@{==>}[ddr]_(.3){H_3^2} \ar@/^14ex/@{-->}[dddr]^(.6){H_3^3} & B_3 \ar[d]^(.45){\beta_3} \\
C_2 \ar[r]_(.4){f_2} \ar[d]_{\alpha_2} \ar@{-->}@/_4ex/[dd]_{h_2^2} \ar@/_1ex/@{==>}[dr]_{H_2^1} \ar@/_3ex/@{-->}[ddr]_(.53){H_2^2} & B_2 \ar[d]^{\beta_2} \\
C_1 \ar[r]_(.53){f_1} \ar[d]_(.4){\alpha_1} & B_1 \ar[d]^(.4){\beta_1} \\
C_0 \ar@{=}[r] & B_0
}.}

As in the $n=2$ case, the fact that the bottom rectangle commutes allows us to simplify the above somewhat. Given a choice of $H_2^1$, one can define $h_2^2:=-\beta_1 H_2^1$ and $H_2^2=0$. Similarly, given a choice of $H_3^2$, one can define $h_3^3:=\beta_1 H_3^2$ and $H_3^3=0$. Thus it suffices to define homotopies
\begin{align*}
h_3^2&:0\simeq\alpha_2\alpha_3, \qquad & \qquad H_3^1&:f_2\alpha_3\simeq\beta_3 f_3, \\
H_2^1&:f_1\alpha_2\simeq\beta_2 f_2,\qquad\text{and} & \qquad H_3^2&: f_1h_3^2+H_2^1\alpha_3+\beta_2 H_3^1\simeq0.
\end{align*}


We now describe explicitly what is needed in order to define these homotopies. Let us begin with $h_3^2$. As in \S\ref{n2case}, the composite $\alpha_2\alpha_3$ vanishes on the $\cA$ and $\cA\{-1\}$ summands, and so a homotopy is required only on the $\cA\{-2\}$ and $\cA\{-3\}$ summands. We will write $U$, $V$, and $W$ for the parameters on the three copies of $X$. As in the $n=2$ case, $\alpha_2\alpha_3$ vanishes on the summands of the form $[U]\cdot[V]$ and $[1-V]\cdot[1-W]$. However, 
\begin{equation*}\begin{split}
\alpha_2\alpha_3([U]\cdot[1-V]) & \\
& \hspace{-4ex}=\alpha_2(([a]\cdot[1-V])_{\mid aX^2}+([U]\cdot[1-V])_{\mid X\Delta_X} \\
& \hspace{4ex}-([U]\cdot[1-V])_{\mid X^2b}) \\
& \hspace{-4ex}= ([a]\cdot[1-a])_{\mid aaX}-([a]\cdot[1-V])_{\mid a\Delta_X}+([a]\cdot[1-V])_{\mid aXb} \\
&\hspace{-4ex}\quad +([a]\cdot[1-V])_{\mid a\Delta_X}-([1-b]\cdot[U])_{\mid Xbb} \\
&\hspace{-4ex}\quad -([a]\cdot[1-V])_{\mid aXb}+([1-b]\cdot[U])_{\mid Xbb} \\
&\hspace{-4ex}= ([a]\cdot[1-a])_{\mid aaX}
\end{split}\end{equation*}
Thus we define 
\[h_3^2([U]\cdot[1-V]):=-(T_{a,1-a})_{\mid aaX}.\]
 Also, we have 
\[h_3^2([1-U]\cdot[V])=-T_{1-a,a}:=T_{a,1-a}\]
 (note that, since $a$ is fixed throughout this discussion, the above choice will not cause compatibility problems), and $h_3^2$ is defined similarly on the other summands.

To see what is needed on the $\cA\{-3\}$ summands, we first note that, for instance
\[\alpha_2\alpha_3([U]\cdot[V]\cdot[1-W])=-[b]\cdot[1-b]\cdot[U].\]
Thus we may define 
\[h_3^2([U]\cdot[V]\cdot[1-W])=-T_{b,1-b}\cdot[U]\]
(note that there is a sign since $[U]\cdot[V]\cdot[1-W]$ is of odd degree). Similarly, one finds that 
\[\alpha_2\alpha_3([U]\cdot[1-V]\cdot[W])=[a]\cdot[1-a]\cdot[W]+[b]\cdot[1-b]\cdot[U].\]
Thus we may define
\[h_3^2([U]\cdot[1-V]\cdot[W])=T_{a,1-a}\cdot[W]+T_{b,1-b}\cdot[U].\]
The nonzero values of $h_3^2$ on the $\cA\{-2\}$ and $\cA\{-3\}$ summands are given as in
 Table~\ref{h32table}.
 
\begin{table}\label{h32table}
\begin{center}
\begin{tabular}{c|c}
\text{generators} & $h_3^2$ \\ \hline
$[U]\cdot[1-V]$ & $-T_{a,1-a}$, \\
$[1-U]\cdot[V]$ & $T_{a,1-a}$, \\
$[V]\cdot[1-W]$ & $T_{b,1-b}$, \\
$[1-V]\cdot[W]$ & $-T_{b,1-b}$, \\ \hline
$[U]\cdot[V][\cdot[1-W]$ & $-T_{b,1-b}\cdot[U]$, \\
$[U]\cdot[1-V]\cdot[W]$ & $T_{a,1-a}\cdot[W]+T_{b,1-b}\cdot[U]$, \\
$[U]\cdot[1-V]\cdot[1-W]$ & $T_{a,1-a}\cdot[1-W]$, \\
$[1-U]\cdot[V]\cdot[W]$ & $-T_{a,1-a}\cdot[W]$, \\
$[1-U]\cdot[V]\cdot[1-W]$ & $-T_{a,1-a}\cdot[1-W]-T_{b,1-b}\cdot[1-U]$, \\
$[1-U]\cdot[1-V]\cdot[W]$ & $T_{b,1-b}\cdot[1-U]$.
\end{tabular}
\end{center}
\caption{Values of $h_3^2$}\end{table}

We do not go into as much detail for $H_3^1$ and $H_2^1$ but merely claim that suitable definitions are
\[\begin{split}
H_3^1([U]\cdot[1-V]\cdot[W])&=-(T_{U,1-U}\cdot[W])_{\mid \Delta_X\times X}+([U]\cdot T_{V,1-V})_{\mid X\times\Delta_X}, \\
H_3^1([U]\cdot[1-V]\cdot[1-W])&=-(T_{U,1-U}\cdot[1-W])_{\mid \Delta_X\times X}, \\
H_3^1([U]\cdot[V]\cdot[1-W])&=-([U]\cdot T_{V,1-V})_{\mid X\times\Delta_X}, 
\qquad\text{and} \\
H_2^1([U]\cdot[1-V])&=T_{U,1-U}.
\end{split}\]
Finally, let us see what is needed to define $H_3^2$. We compute
\begin{equation*}\begin{split}
(f_1h_3^2+H^1_2\alpha_3+\beta_2H_3^1)([U]\cdot[V]\cdot[1-W]) & = \\
& \hspace{-17em} -T_{b,1-b}\cdot[U] +[a]\cdot T_{U,1-U} -[a]\cdot T_{U,1-U}+[U]\cdot T_{U,1-U}-[U]\cdot T_{b,1-b} \\
& = [U]\cdot T_{U,1-U}
\end{split}\end{equation*}
Thus $H_3^2([U]\cdot[V]\cdot[1-W])$ is given by $-T_{U,U,1-U}$, where $d(T_{U,U,1-U})=[U]\cdot T_{U,1-U}$. Note that since $[U]^2=0$, the class $[U]\cdot T_{U,1-U}$ is a representative for the Massey product $\langle[U],[U],[1-U]\rangle$.
Similarly, 
\[(f_1h_3^2+H^1_2\alpha_3-\beta_2H_3^1)([U]\cdot[1-V]\cdot[W]) = T_U\cdot[U]-[U]\cdot T_U=-2\cdot[U] \cdot T_U.\]
The class $H_3^2([U]\cdot[1-V]\cdot[W])$ is then given by a class $2T_{U,U,1-U}$. Note that since $T_{1-U,U}=-T_{U,1-U}$, we have that $2[U]\cdot T_{U,1-U}$ represents the Massey product $-\langle[U],[1-U],[U]\rangle$.  


We have sketched a proof of the following

\begin{prop}In order to build the $n+1$-truncated motivic fundamental group ring of $X=\bP^1-\{0,1,\infty\}$ in the case $n=3$, it suffices that the triple Massey products of $[t]$ and $[1-t]$ in 
\[\rH^{2,3}(X;\Q)\iso\CH^3(X,4)\otimes\Q\]
are defined and contain zero.
\end{prop}

\subsection{The general case}\label{gensec}

From the discussion in the previous sections, it is not difficult to guess at the picture for general $n$. One starts with the diagram

\centerline{\xymatrix {
B_n \ar[r]^{\beta_n} & B_{n-1} \ar[r]^{\beta_{n-1}} & \dots \ar[r] & B_1 \ar[r]^{\beta_1} & B_0 \\
C_n \ar[u]^{f_n} \ar[r]^{\alpha_n} & C_{n-1} \ar[u]^{f_{n-1}} \ar[r]^{\alpha_{n-1}} & \dots \ar[r] & C_1 \ar[u]^{f_1} \ar[r]^{\alpha_1} & C_0 \ar@{=}[u],
}}

\vspace{1mm}
\noindent where the top row is the complex~(\ref{theMotivicComplex}), each $C_i$ is a free $\cA$-module, and each $f_i$ is a quasi-isomorphism. 
In addition, one needs homotopies 
\[h_i^j:C_i\to C_{i-j}[1-j] \qquad \text{ for each } 3\leq i\leq n \quad \text{ and } 2\leq j\leq n-1\]
and 
\[H_k^l:C_k\to B_{k-l}[-l] \qquad \text{for each } 2\leq k\leq n \quad \text{ and }1\leq l\leq n-1.\]
As in the previous cases, the right-hand square commutes on the nose, and no homotopies are required in order to deal with it. 

In order to build the $h^2_i$'s and the $H^1_j$'s, one only needs the (ordinary) Totaro cycles. In order to build the $h_i^3$'s and the $H_j^2$'s, one needs in addition certain triple Massey products to be defined and to contain zero. In order to define the $h_i^4$'s and the $H_j^3$'s, one needs certain $4$-fold Massey products to be defined and to contain zero. This pattern continues, so that in order to build the $h_1^{n-2}$ and the $H_j^{n-3}$'s, one needs certain $(n-2)$-fold Massey products to be defined and to contain zero, and in order to define $H_1^{n-2}$ one needs certain $(n-1)$-fold Massey products to be defined and to contain zero. Precise statements are given below.

In order to give precise statements, it will be convenient to introduce some notation. Given a function $f$ on $\G_m-\{1\}$, we will write $\Phi_f$ to mean either $[f]$ or $[1-f]$. Thus the product $\Phi^1_U\Phi^2_V$ can mean any of the following four cohomology classes:$[U][V]$, $[1-U][V]$, $[U][1-V]$, or $[1-U][1-V]$. Note that a set of generators of $C_n$ is given by the elements of the form $\Phi^1_{U_1}\cdots\Phi^n_{U_n}$.

\begin{prop}\label{CompMasseyProp}For $2\leq k\leq n$, the homotopy $h_n^k:C_n\to C_{n-k}[1-k]$ satisfies
\[\begin{split}
d_{n-k}(h_n^k\Phi^1_{U_1}\cdots\Phi^n_{U_n}) &= \\
 &\hspace{-15ex}(-1)^\lambda\bigl(\langle\Phi^1_a,\dots,\Phi^k_a\rangle \cdot \Phi^{k+1}_{U_{k+1}}\cdots\Phi^n_{U_n}-\langle \Phi^{n-k+1}_b\dots\Phi^n_b\rangle \cdot \Phi^1_{U_1}\cdots\Phi^{n-k}_{U_{n-k}}\bigr),
\end{split}\]
where $\lambda=n(k-1)+\frac{k(k-1)}2$. Thus $h_n^k$ can be defined by
\[\begin{split}
h_n^k(\Phi^1_{U_1}\cdots\Phi^n_{U_n}) &= \\
&\hspace{-10ex}(-1)^\lambda\bigl( T_{\Phi^1_a,\dots,\Phi^k_a}\cdot\Phi^{k+1}_{U_k}\cdots\Phi^n_{U_n}-T_{\Phi^{n-k+1}_b,\dots,\Phi^n_b}\cdot\Phi^1_{U_1}\cdots\Phi^{n-k}_{U_{n-k}}\bigr).
\end{split}\]
\end{prop}

\begin{proof}The proof is by double induction on $k$ and $n$. The base case is $k=n=2$, and we have
\[d_0h_2^2(\Phi^1_{U_1}\Phi^2_{U_2})=-\alpha_1\alpha_2(\Phi^1_{U_1}\Phi^2_{U_2})=-(\Phi^1_a\Phi^2_a-\Phi^1_b\Phi^2_b)\]
as desired. In fact, we similarly find that
\[d_{n-2}h_n^2(\Phi^1_{U_1}\dots\Phi^n_{U_n})=(-1)^{n-1}\bigl(\Phi^1_a\Phi^2_a\cdot\Phi^3_{U_3}\cdots\Phi^n_{U_n}-\Phi^{n-1}_b\Phi^n_b\cdot\Phi^1_{U_1}\cdots\Phi^{n-2}_{U_{n-2}}\bigr)\]
for any $n$.

For the inductive step, assume the formula holds for all $h^i_j$ with $j<n$ or $i<k$. Recall the matrices $C$, $B$, and $H$ from \S\ref{n3case}. The matrix equation $HC=BH$ gives the equation
\[\begin{split}
d_{n-k}h^k_n\Phi^1_{U_1}\cdots\Phi^n_{U_n} & =-(-1)^n\bigl( h_{n-1}^{k-1}\alpha_n\Phi^1_{U_1}\cdots\Phi^n_{U_n} - h_{n-2}^{k-2}h_n^2\Phi^1_{U_1}\cdots\Phi^n_{U_n}+ \\
&\qquad\qquad\qquad \cdots+(-1)^k\alpha_{n-k+1}h^{k-1}_n\Phi^1_{U_1}\cdots\Phi^n_{U_n}\bigr).
\end{split}\]
By induction, we know all of the terms in this sum. The general term is given by
\[\begin{split}
h^{k-j}_{n-j}h^j_n\bigl(\Phi^1_{U_1}\cdots\Phi^n_{U_n}\bigr)=&(-1)^{nk+j+\frac{k(k-1)}2} \cdot \\
& \hspace{-24ex}\biggl( T_{\Phi^1_a,\dots,\Phi^j_a}\bigl( T_{\Phi^{j+1}_a,\dots,\Phi^k_a}\Phi^{k+1}_{U_{k+1}}\cdots\Phi^n_{U_n}-T_{\Phi^{n-k+j+1}_b,\dots,\Phi^n_b}\Phi^{j+1}_{U_{j+1}}\cdots\Phi^{n-k+j}_{U_{n-k+j}}\bigr) \\
& \hspace{-24ex}-T_{\Phi^{n-j+1}_b,\dots,\Phi^n_b}\bigl( T_{\Phi^1_a,\dots,\Phi^{k-j}_a}\Phi^{k-j+1}_{U_{k-j+1}}\cdots\Phi^{k-j}_{U_{k-j}}-T_{\Phi^{n-k+1}_b,\dots,\Phi^{n-j}_b}\Phi^1_{U_1}\cdots\Phi^{n-k}_{U_{n-k}}\bigr)\biggr)
\end{split}\]
One checks that the terms involving a product of a Totaro cycle at $a$ with a Totaro cycle at $b$ cancel the similar mixed terms appearing in the expression for 
$h^j_{n-k+j}h^{k-j}_n(\Phi^1_{U_1}\Phi^n_{U_n})$, using that $h^{k-j}_{n-j}h^j_n$ appears with sign $(-1)^{n+j}$. Summing the terms gives the result.
\end{proof}

\begin{eg}\label{EgCompMasseyProp} We give some examples to illustrate the result. Proposition~\ref{CompMasseyProp} gives, for instance, the following equations
\[\begin{split}
h_4^2\bigl([U][1-V][W][1-X]\bigr) &=-\left(T_{a,1-a}[W][1-X]-T_{b,1-b}[U][1-V]\right), \\
h_4^3\bigl([U][1-V][W][1-X]\bigr) &= -\left(T_{a,1-a,a}[1-X]-T_{1-b,b,1-b}[U]\right), \\
h_5^2\bigl([U][1-V][W][1-X][1-Y]\bigr)&= \\
 &\hspace{-8em}\left(T_{a,1-a}[W][1-X][1-Y]-T_{1-b,1-b}[U][1-V][W] \right) \\
 &= T_{a,1-a}[W][1-X][1-Y], \\
h_5^3\bigl([U][1-V][W][1-X][1-Y]\bigr)&= \\
 &\hspace{-8em}-\left(T_{a,1-a,a}[1-X][1-Y]-T_{b,1-b,1-b}[U][1-V] \right), \quad\text{and}\\
h_5^4\bigl([U][1-V][W][1-X][1-Y]\bigr)&= \\
 &\hspace{-8em}-\left(T_{a,1-a,a,1-a}[1-Y]-T_{1-b,b,1-b,1-b}[U] \right).
\end{split}\]
We will use these formulae in \S\ref{TrimFatSect} to describe the associated minimal modules.
\end{eg}

One can similarly obtain inductively a formula for the $H_n^k$'s. In order to make a precise statement, it is convenient to introduce the following definition.

\begin{defn}We will say that the {\bf length} of an $n$-fold Massey product is $n$. Similarly, if $T_{x_1,\dots,x_n}$ is a Totaro cycle bounding an $n$-fold Massey product, we will say that it is a Totaro cycle of {\bf length} $n$. The {\bf reduced length} of a Massey product or Totaro cycle will be defined to be the length minus $1$.
\end{defn}

\begin{prop}For $1\leq k\leq n-1$, the homotopy $H_n^k:C_n\to B_{n-k}[-k]$ can be defined by
\[
H_n^k(\Phi^1_{U_1}\cdots\Phi^n_{U_n}) = 
(-1)^\mu \biggl( \sum \Phi^1_{U_1}\cdots T_{\Phi^i_{U_i},\dots,\Phi^j_{U_i}} \Phi^{j+1}_{U_{j+1}}\cdots T_{\Phi^k_{U_k}\cdots\Phi^l_{U_k}}\cdots\Phi^n_{U_n}\biggr)
\]
where $\mu=nk+\frac{k(k-1)}2$ and where the sum is over all such products such that the sum of the reduced lengths of the Totaro cycles is $k$. 
\end{prop}

\begin{proof}
We do not give the complete proof, as it is very similar to the proof of Proposition~\ref{CompMasseyProp}. Again the argument is by a double induction. The base case is easy to verify, and for the induction step one uses the equation
\[dH^i_j=(-1)^{i+j}\beta_{j-i+1}H^{i-1}_j+H^i_jd+(-1)^j\left(H^{i-1}_{j-1}\alpha_j+\dots+f_{j-i}h^i_j\right),\]
which once again comes from the matrix equation $HC=BH$. One then checks that all terms involving $a$'s or $b$'s cancel.
\end{proof}

\begin{eg} As in Example~\ref{EgCompMasseyProp}, we give some examples to illustrate the result. The proposition yields the equations
\[\begin{split}
H_4^2\bigl([U][1-V][W][1-X]\bigr) &=-T_{U,1-U}T_{W,1-W}-T_{U,1-U,U}[1-X] \\
 & \hspace{5em} -[U]T_{1-V,V,1-V} \\
H_4^3\bigl([U][1-V][W][1-X]\bigr) &=-T_{U,1-U,U,1-U} \\
H_5^2\bigl([U][1-V][W][1-X][1-Y]\bigr) &=-T_{U,1-U}T_{W,1-W}[1-Y] \\
& \hspace{-5em}-T_{U,1-U}[W]T_{1-X,1-X}-[U]T_{1-V,V}T_{1-X,1-X} \\
&\hspace{-5em}-T_{U,1-U,U}[1-X][1-Y]-[U]T_{1-V,V,1-V}[1-Y] \\
& \hspace{3em} -[U][1-V]T_{W,1-W,1-W} \\
H_5^3\bigl([U][1-V][W][1-X][1-Y]\bigr) &=T_{U,1-U}T_{W,1-W,1-W}+T_{U,1-U,U}T_{1-U,1-U} \\
 & \hspace{-5em}+T_{U,1-U,U,1-U}[1-Y]+[U]T_{1-V,V,1-V,1-V}, \text{ and} \\
 H_5^4\bigl([U][1-V][W][1-X][1-Y]\bigr) &=T_{U,1-U,U,1-U,1-U}
\end{split}\]
\end{eg}

The above proposition directly implies the following theorem.


\begin{thm}
In order to build the $n$-truncated motivic fundamental group ring of $X=\bP^1-\{0,1,\infty\}$ for general $n$, it suffices that the iterated Massey products$\langle \Phi^1_t,\dots,\Phi^k_t\rangle$ in 
\[\rH^{2,n}(X;\Q)\iso\CH^n(X,2n-2)\otimes\Q\]
are defined and contain zero.
\end{thm}

Note that the fact that the $n$-fold iterated Massey product are defined implies in particular that the $k$-fold iterated Massey products are defined and contain zero for $k<n$.


\subsection{Trimming the fat}\label{TrimFatSect}

Rather than just obtaining a cellular approximation to the complex~(\ref{theMotivicComplex}), one would like to obtain an approximation which is a minimal $\cA$-module (\S\ref{MinModSect}). In order to do this, we must first ensure that $\cA$ is minimal, meaning that the differential lands in the decomposable elements. 
By Theorem~IV.2.4 of \cite{KM}, there is a quasi-isomorphism $\varphi:\cA_{min}\xrightarrow{\sim}\cA$, and the minimal model $\cA_{min}$ can be chosen so that $\cA_{min}^{1,1}=k^\times\otimes\Q$.
By Proposition~\ref{DerivedInvarProp} the pullback functor $\varphi^*:\sD(\cA)\to\sD(\cA_{min})$ is an equivalence of triangulated categories, so that the $t$-structure is carried along as well. 

The pullback $\varphi^*(C_i)$ of each free $\cA$-module will not be a free $\cA_{min}$-module, but we may replace each $\varphi^*(C_i)$ by a free $\cA_{min}$-module on generators in appropriate degrees. 
Since the maps $\alpha_n:C_n\to C_{n-1}$ send the generator of a cell to a sum of generators of cells and cocycles $[a]$ and $[b]$, and since $\cA_{min}^{1,1}=k^\times\otimes\Q$, it is clear how to define a complex of free $\cA_{min}$-modules modeling the complex $C_*$ of free $\cA$-modules.
We thus assume from now on that $\cA$ is a minimal dga.

 The explicit cellular approximations we have obtained lend themselves easily to finding minimal cellular approximations. We do not go into full details, but the resulting picture is 

\centerline{\xymatrix {
B_n \ar[r] & B_{n-1} \ar[r] & \dots \ar[r] & B_1 \ar[r] & B_0 \\
C_n \ar[r] \ar[u]^{\sim} \ar@{->>}[d] & C_{n-1} \ar[r] \ar[u]^{\sim} \ar@{->>}[d] & \dots \ar[r] & C_1 \ar[r] \ar[u]^{\sim} \ar@{->>}[d] & C_0 \ar@{=}[u] \ar@{->>}[d] \\
\cA\{-n\}^{\oplus 2^n} \ar[r] & \cA\{-(n-1)\}^{\oplus 2^{n-1}} \ar[r] & \dots \ar[r] & \cA\{-1\}^{\oplus 2} \ar[r] & \cA.
}}

\noindent The homotopies $h_i^k$ descend to the quotients and allow one to build a minimal cell module out of 
\[\cA\{-n\}^{\oplus 2^n}\to\dots\to\cA\{-1\}^{\oplus 2}\to\cA,\]
 and the epimorphisms \mbox{$C_i\twoheadrightarrow\cA\{-i\}^{2^i}$} assemble to give a quasi-isomorphism of cell modules. Ignoring the differential, the minimal module is given by 
 \[\cA(-n)^{\oplus 2^n}\oplus\cA(-n+1)^{\oplus 2^{n-1}}\oplus\dots\oplus \cA(-1)^{\oplus 2}\oplus \cA.\]
 In particular, it lives in $\cH=\sD(\cA)^{\geq 0}\cap\sD(\cA)^{\leq 0}$ and thus defines a mixed Tate motive.

\begin{eg} Let $n=1$. The cell module is the cofiber of 
\[\cA\{-1\}^{\oplus 2}\oplus \cA\to \cA^{\oplus 2}.\]
To obtain a minimal module, it suffices to kill the summand $\cA$ on the left, together with its image. Forgetting the differential, the minimal module is given by $\cA\oplus\cA(-1)^2$. The differential is given by the matrix
\[\left(\begin{array}{ccc} d & [b]-[a] & [1-b]-[1-a] \\ 0 & d & 0 \\ 0 & 0 & d
\end{array}\right)\]
In fact, letting $M_1$ denote the minimal module, we have an extension
\[\cA\to M_1\to \cA(-1)^{\oplus 2}\]
in $MTM_{KM}$.
\end{eg}

\begin{eg} Let $n=2$. The cell module is built from

\centerline{\xymatrix  @R=3ex {
\cA\{-2\}^{\oplus 4}\oplus\cA\{-1\}^{\oplus 4}\oplus\cA \ar[r] \ar@{=}[d] &\left(\cA\{-1\}^{\oplus 2}\oplus\cA\right)^{\oplus 3} \ar[r] \ar@{=}[d] & \cA^{\oplus 3}.  \ar@{=}[d] \\
C_2 & C_1 & C_0
}}

\noindent To build the associated minimal module, we first quotient out the copy of $\cA$ in $C_2$, which kills one of the copies in $C_1$. We then quotient the remaining copies of $\cA$ in $C_2$, which leaves only one copy in $C_0$. Lastly, we factor out the copies of $\cA\{-1\}$ in $C_2$, which leaves two copies in $C_1$. This produces a minimal module which, ignoring the differential, is given by $\cA\oplus\cA(-1)^{\oplus 2}\oplus\cA(-2)^{\oplus 4}$. Choosing the basis 
\[\{1, [U],[1-U],[U][V],[U][1-V],[1-U],[V],[1-U][1-V],\}\]
we have that the differential is given by the matrix in Figure~\ref{n2matrix}.
\begin{figure}
\[\left(\begin{array}{cccccccc} 
d & [b]-[a] & \begin{array}{c}[1-b] \\ \hspace{2ex}-[1-a]\end{array} & 0 & \begin{array}{c}T_{b,1-b} \\ \hspace{2ex}-T_{a,1-a}\end{array} & \begin{array}{c}T_{1-b,b} \\ \hspace{2ex}-T_{1-a,a} \end{array} & 0 \\
0 & d & 0 & [a]-[b] & -[1-b] & [1-a] & 0 \\
0 & 0 & d & 0 & [a] & -[b] & \begin{array}{c}[1-a] \\ \hspace{2ex}-[1-b]\end{array} \\
0 & 0 & 0 & d & 0 & 0 & 0 \\
0 & 0 & 0 & 0 & d & 0 & 0 \\
0 & 0 & 0 & 0 & 0 & d & 0 \\
0 & 0 & 0 & 0 & 0 & 0 & d
\end{array}\right)\]
\caption{The differential for the minimal module, $n=2$.}\label{n2matrix}
\end{figure}
\end{eg}

\begin{eg} One can give a similar description of the minimal module when $n=3$. Forgetting the differential, it is given by 
\[\cA\oplus\cA(-1)^{\oplus 2}\oplus\cA(-2)^{\oplus 4}\oplus\cA(-3)^{\oplus 8}.\]
We extend our previously chosen basis to a basis here by taking the elements 
\[\begin{split}
[U][V][W], &[U][V][1-W], [U][1-V][W], [U][1-V][1-W], [1-U][V][W], \\
&[1-U][V][1-W], [1-U][1-V][W], [1-U][1-V][1-W].\end{split}\]
 The differential is then given by the matrix in Figure~\ref{n3matrix}.
\begin{sidewaysfigure}
\scalebox{0.75}{
\addtolength{\arraycolsep}{-0.6ex}
$\left(\begin{array}{ccccccccccccccc} 
d \vspace{2ex} & [b]-[a] & \begin{array}{c}[1-b] \\ \hspace{2ex}-[1-a]\end{array} & 0 & \begin{array}{c}T_{b,1-b} \\ \hspace{2ex}-T_{a,1-a}\end{array} & \begin{array}{c}T_{1-b,b} \\ \hspace{2ex}-T_{1-a,a} \end{array} & 0 & 0 & \begin{array}{c} T_{a,a,1-a} \\ \hspace{-1ex}-T_{b,b,1-b}\end{array} & \begin{array}{c} T_{a,1-a,a} \\ \hspace{-1ex}-T_{b,1-b,b} \end{array} & \begin{array}{c} T_{a,1-a,1-a} \\ \hspace{-1ex}-T_{b,1-b,1-b}\end{array} & \begin{array}{c} T_{1-a,a,a} \\ \hspace{-1ex}-T_{1-b,b,b}\end{array} & \begin{array}{c} T_{1-a,a,1-a} \\ \hspace{-1ex} -T_{1-b,b,1-b} \end{array} & \begin{array}{c} T_{1-a,1-a,a} \\ \hspace{-1ex}-T_{1-b,1-b,b} \end{array} & 0 \\
0 \vspace{2ex}  & d & 0 & [a]-[b]  & -[1-b] & [1-a] & 0 & 0 & T_{b,1-b} & \begin{array}{c} T_{1-b,b} \\ \hspace{-1ex}-T_{a,1-a} \end{array} & 0 & T_{1-a,a} & 0 & 0 & 0 \\
0 \vspace{2ex}   & 0 & d & 0 & [a] & -[b] & \begin{array}{c}[1-a]  \\ \hspace{2ex}-[1-b]\end{array} & 0 & 0 & 0 & -T_{a,1-a} & 0 & \begin{array}{c} T_{b,1-b} \\ \hspace{-1ex}-T_{1-a,a} \end{array} & T_{1-b,b} & 0 \\
0 \vspace{2ex}   & 0 & 0 & d & 0 & 0 & 0 & [b]-[a] & [1-b] & 0 & 0 & -[1-a] & 0 & 0 & 0 \\
0 \vspace{2ex}   & 0 & 0 & 0 & d & 0 & 0 & 0 & -[a] & [b] & [1-b] & 0 & -[1-a] & 0 & 0 \\
0 \vspace{2ex}   & 0 & 0 & 0 & 0 & d & 0 & 0 & 0 & -[a] & 0 & [b] & [1-b] & -[1-a] & 0 \\
0 \vspace{2ex} & 0 & 0 & 0 & 0 & 0 & d & 0 & 0 & 0 & -[a] & 0 & 0 & [b] & \begin{array}{c}[1-b] \\ \hspace{2ex}-[1-a]\end{array} \\
0 \vspace{2ex}  & 0 & 0 & 0 & 0 & 0 & 0 & d & 0 & 0 & 0 & 0 & 0 & 0 & 0 \\
0 \vspace{2ex}  & 0 & 0 & 0 & 0 & 0 & 0 & 0 & d & 0 & 0 & 0 & 0 & 0 & 0 \\
0 \vspace{2ex}  & 0 & 0 & 0 & 0 & 0 & 0 &  0 & 0 & d & 0 & 0 & 0 & 0 & 0 \\
0 \vspace{2ex}   & 0 &  0 & 0 & 0 & 0 & 0 & 0 & 0 & 0 & d & 0 & 0 & 0 & 0 \\
0 \vspace{2ex}  & 0 & 0 & 0 & 0 & 0 & 0 &  0 & 0 & 0 & 0 & d & 0 & 0 & 0 \\
0 \vspace{2ex}  & 0 & 0 & 0 & 0 & 0 & 0 &  0 & 0 & 0 & 0 & 0 & d & 0 &  0 \\
0 \vspace{2ex}  & 0 &  0 & 0 & 0 & 0 & 0 & 0 & 0 & 0 & 0 & 0 & 0 & d & 0 \\
0 & 0 & 0 & 0 &  0 & 0 & 0 & 0 & 0 & 0 & 0 & 0 & 0 & 0 & d
\end{array} \right) $ }
\caption{The differential for the minimal module, $n=3$.}\label{n3matrix}
\end{sidewaysfigure}
\end{eg}

\subsection{Vanishing of Massey Products}\label{VanshMssySect}

Recall that we have the following noncommutative diagram of dg-$\cA$-modules:

\centerline{\xymatrix {
B_n \ar[r] & B_{n-1} \ar[r] & \dots \ar[r] & B_1 \ar[r] & B_0 \\
C_n \ar[r] \ar[u]^{\sim} & C_{n-1} \ar[r] \ar[u]^{\sim} & \dots \ar[r] & C_1 \ar[r] \ar[u]^{\sim} & C_0. \ar@{=}[u] \\
}}

\noindent The object in which we have been interested is the totalization, or convolution in the sense of \S{IV.2} of \cite{GM},
of the complex $B_*$ in the triangulated category $\sD(\cA)$. Since $B_*$ is an honest complex of dg-$\cA$-modules, we can produce a totalization by taking the total complex of the associated double complex. We have approximated $B_*$ by $C_*$, in the sense that the vertical maps are quasi-isomorphisms. Thus $C_*$ is isomorphic to $B_*$ in $\sD(\cA)$, and it follows that $C_*$ has a totalization. At this point, we are using commutativity of the diagram in $\sD(\cA)$, which only requires the homotopies $H^1_i$. Recall from above that the ordinary Steinberg relation suffices to define these.

It is well-known (cf. \S{IV.2} of \cite{GM}) that $C_*$ has a totalization if and only if the Toda bracket of the maps in $C_*$ is defined and contains $0$. 
Since $\sD(\cA)\simeq h\Cell_{\cA}$ (Theorem~\ref{CellEqDer}), 
Toda brackets in $\sD(\cA)$ can be computed as compositional Massey products (\cite{BoKa}). We thus have that these compositional Massey products contain $0$.
Finally, our Proposition~\ref{CompMasseyProp} gives an explicit description of these compositional Massey products in terms of Massey products of the Steinberg symbols $[t]$ and $[1-t]$. Despite the above discussion, it does not follow that the Massey products of Steinberg symbols are defined and contain zero; rather, we can only conclude that {\em some} representative for the compositional Massey product is null.

Note also that if the ground field $k$ is a number field, 
then according to Example~\ref{CohomNumFldEg} the groups $\rH^{2,*}(\Spec k;\Q)$ vanish. This guarantees that a Massey product of $[a]$ and $[1-a]$ will vanish {\em as long as it is defined}.


\subsection{Bloch-Totaro cycles}\label{hightotsec}

Here and in \S\ref{34MasseyProds},
we discuss bounding cochains for the Massey products of $[a]$'s and $[1-a]$'s. For this purpose it is more convenient to work with Bloch's alternating cycle complex (\S\ref{altcyclsec}) $A$ rather than the Suslin-Voevodsky variant $\cA$. The construction of cell modules, and therefore also of minimal modules, given above is equally valid over $A$ as over $\cA$, so there is no loss in making the transition.

Recall that in his thesis \cite{T}, Totaro wrote down cochains $T_{a,1-a}\in A^{1,1}$ bounding the elements $[a]\cdot[1-a]$. In \S\ref{BlchCyclSect}, we wrote $\square=\A^1$, but for the purposes of this section it is convenient to use the identification $\A^1\iso \bP^1-\{1\}$ given by $x\mapsto \frac{x}{x-1}$. We thus agree to write $\square$ for $\bP^1-\{1\}$ from now on. Note that under this identification, the map $k^\times\to A^{1,1}$ is given by $a\mapsto a\subset\square$. 

Recall (\S\ref{altcyclsec}) that $A^{p,q}$ is a quotient of the $\Q$-vector space with generators given by closed, irreducible, admissible, codimension $q$ subvarieties $W\subseteq \square^{2q-p}$. The Totaro cycle $T_{a,1-a}\in A^{1,2}$ is given by the subvariety
\[ \left[x,1-x,1-\frac{a}x\right]\subseteq\square^3.\]
Later, Bloch (\cite{B2}, \S3) generalized this construction to give cycles
\[ \rho_n(a)=\left[x_1,1-x_1,1-\frac{x_2}{x_1},x_2,1-\frac{x_3}{x_2},\dots,
x_{n-1},1-\frac{a}{x_{n-1}}\right]\subseteq\square^{2n-1}\]
in $A^{1,n}$. One has $\rho_2(a)=T_{a,1-a}$ and $d\rho_n(a)=\rho_{n-1}(a)\cdot a$. Thus we see that $\rho_n(a)$ is a model for $T_{a,a,\dots,a,1-a}$, a bounding cochain for the Massey product $\langle a, a, \dots, 1-a\rangle$. 

\subsection{$3$-fold and $4$-fold Massey products}\label{34MasseyProds}

As the automorphism $z\mapsto 1-z$ of
\mbox{$\bP^1-\{0,1,\infty\}$} 
takes $[t]$ to $[1-t]$, it suffices to consider Massey products whose first factor is $[a]$. 

For $n=3$ the Massey product $\langle a, a, a\rangle$ is trivially zero, and we have seen that the Bloch-Totaro cycle gives a bounding cochain $T_{a,a,1-a}$ for $\langle a, a, 1-a\rangle$. We similarly have a bound for $\langle a, 1-a, 1-a\rangle=-\langle 1-a, 1-a, a\rangle$. Finally, $\langle a, 1-a, a\rangle $ is represented by a cocyle of the form 
\[T_{a,1-a}[a]+[a]T_{1-a,a}=-2[a]T_{a,1-a},\]
 so $-2T_{a,a,1-a}$ bounds $\langle a, 1-a, a\rangle$.

For $n=4$ the Massey product $\langle a, a, a, 1-a\rangle$ is again killed by a Bloch-Totaro cycle $T_{a,a,a,1-a}$. The relation $T_{a,1-a,a}=-2T_{a,a,1-a}$ shows that $-3[a]T_{a,a,1-a}$ represents the Massey product $\langle a, a, 1-a, a\rangle$, so that again the Bloch-Totaro cycle can be used to bound it. Symmetry gives 
\[\langle a, 1-a, a, a\rangle =-\langle a, a, 1-a, a\rangle\] 
and 
\[\langle a, 1-a, 1-a, a\rangle=-\langle a, 1-a, 1-a, a\rangle,\]
so that the latter is trivially zero. Of the remaining cases, 
\[\langle a, 1-a, 1-a, 1-a\rangle =-\langle 1-a, 1-a, 1-a, a\rangle\]
is killed by a Bloch-Totaro cycle, and 
\[\langle a, 1-a, a, 1-a\rangle = 2\langle a, a, 1-a, 1-a\rangle,\]
so it remains only to deal with $\langle a, a, 1-a, 1-a\rangle$. One possible representative for the latter Massey product is given by
\begin{equation}\begin{split}
&\frac12 \left[1-\frac{a}t,t,1-\frac{t}u,u,1-u\right]\cdot[1-a] \\
&+ \frac12\left[1-\frac{1-a}{1-t},t,1-\frac{1-t}{1-u},u,1-u\right]\cdot[1-a] \\
&+ \frac12[a]\cdot\left[1-\frac{1-a}t,1-\frac{t}u,1-u,u,t\right] \\
&+ \frac12[a]\cdot\left[1-\frac{a}{1-t},1-\frac{1-t}{1-u},1-u,u,t\right].
\end{split}\notag\end{equation}
This defining cocycle was found by using the methodology of \cite{GGL} and \cite{FJ}, but we have not been able to bound this Massey product using their methods. We have hope that the ``binary'' case of \cite{Jaf} will produce bounding cochains for the above $n$-fold Massey products, though at present the algebraic cycles that are produced in loc. cit. do not all satisfy the admissibility requirement.






\bibliographystyle{alphanum}
\bibliography{ThesisRefs}

\end{document}